\documentclass[preprint,12pt]{elsarticle}

\usepackage{amssymb}
\usepackage{amsmath}
\usepackage{amsthm}
\usepackage{natbib}
\usepackage[margin=1.1in]{geometry}

\journal{Theoretical Computer Science}

\usepackage{mathtools}
\usepackage[utf8]{inputenc} 
\usepackage[T1]{fontenc}    
\usepackage{hyperref}       
\usepackage{url}            
\usepackage{booktabs}       
\usepackage{amsfonts}       
\usepackage{nicefrac}       
\usepackage{microtype}      
\usepackage{xcolor}         
\usepackage{graphicx}
\usepackage{adjustbox}
\usepackage{amssymb, bm}
\usepackage{stmaryrd}
\usepackage{bbm}
\usepackage{float}
\usepackage{soul}

\newcommand{\R}{\mathbb{R}}
\newcommand{\TA}{\mathcal{A}}

\newcommand{\matrixnotation}[1]{\mathbf{#1}}

\newcommand{\A}{\matrixnotation{A}}

\newcommand{\vectornotation}[1]{\mathbf{#1}}
\newcommand{\x}{\vectornotation{x}}
\newcommand{\y}{\vectornotation{y}}
\newcommand{\g}{\vectornotation{g}}
\renewcommand{\a}{\vectornotation{a}}
\renewcommand{\b}{\vectornotation{b}}
\renewcommand{\v}{\vectornotation{v}}

\newcommand{\M}{\mathbf{M}}

\newcommand{\B}{\mathbf{B}}

\newcommand{\E}{E}
\newcommand{\Var}{\mathrm{Var}}
\newcommand{\Cov}{\mathrm{Cov}}
\newcommand{\En}[1]{E\left[#1\right]}
\newcommand{\Expt}[1]{E\left[#1\right]}

\newcommand{\Varn}[1]{\mathrm{Var}\left(#1\right)}

\newcommand{\Prn}[1]{\mathrm{Pr}\left(#1\right)}

\newcommand{\tr}{\mathrm{tr}}
\newcommand{\trn}[1]{\mathrm{tr}\left(#1\right)}

\newcommand{\diagn}[1]{\mathrm{diag}\left(#1\right)}

\newcommand{\vtimes}{\bar{\times}}

\newcommand{\e}{\mathbf{e}}

\newcommand{\modeororder}{order}
\newcommand{\nmodeororder}{$N$-\modeororder}
\newcommand{\anordertensor}{an  \nmodeororder~tensor}

\newcommand{\rademacher}{Rademacher}
\newcommand{\Rademacher}{\rademacher}
\newcommand{\iidrade}{\textit{i.i.d.} \rademacher}

\newcommand{\ie}{\textit{i.e.}}
\newcommand\numberthis{\addtocounter{equation}{1}\tag{\theequation}}

\usepackage{mwe}

\newcommand{\RVar}{z}
\usepackage{thmtools, thm-restate}
\usepackage{mathtools}
\usepackage{paralist}
\usepackage{wrapfig}
\usepackage{caption}
\usepackage{float}

\allowdisplaybreaks

\newtheorem{theorem}{Theorem}

\newtheorem{definition}[theorem]{Definition}
\newtheorem{lem}[theorem]{Lemma}
\newtheorem{cor}[theorem]{Corollary}

\usepackage{lineno,hyperref}

\begin{document}

\begin{frontmatter}
\title{Stochastic Trace and Diagonal Estimator for Tensors}
\author[aff]{Bhisham Dev Verma}
\ead{bhishamdevverma@gmail.com}
\author[aff1]{Rameshwar Pratap}
\ead{rameshwar@cse.iith.ac.in}
\author[aff2]{Keegan Kang}
\ead{kk054@bucknell.edu}

\affiliation[aff]{organization={Wake Forest University},
            city={Winston-Salem},
            state={North Carolina},
            country={USA}}
\affiliation[aff1]{organization={Indian Institute of Technology Hyderabad},
            state={Telangana},
            country={India}}

\affiliation[aff2]{organization={Bucknell University},
            city={Lewisburg},
            state={ Pennsylvania},
            country={USA}}

\begin{abstract}
We consider the problem of estimating the trace and diagonal entries of \anordertensor~(where $N \geq 2$) under the framework where the tensor can only be accessed through tensor-vector multiplication. The aim is to estimate the tensor's diagonal entries and trace by minimizing the number of tensor-vector queries. The seminal work of Hutchinson~\citep{hutchinson1989stochastic}, and    \citep{bekas2007estimator} give unbiased estimates of the trace and diagonal elements of a given matrix, respectively, using matrix-vector queries. However, to the best of our knowledge, no analogous results are known for estimating the trace and diagonal entries of higher-order tensors using tensor-vector queries.
This paper addresses this gap and presents unbiased estimators for the trace and diagonal entries of a tensor under this model. Our proposed methods can be seen as generalizations of~\citep{hutchinson1989stochastic,bekas2007estimator}, and reduce to their estimators for the matrix when $N=2$. We provide a rigorous theoretical analysis of our proposals and complement it with supporting simulations.
\end{abstract}


\end{frontmatter}

\section{Introduction}
The trace and diagonal entries of a matrix $\A \in \mathbb{R}^{d\times d}$ are used in a wide range of applications in various fields, such as triangle counting in graphs ~\citep{avron2010counting}, computing the Estrada index of a graph~\citep{de2007estimating}, quantum chromodynamics~\citep{thron1998pade}, computing the log-determinant~\citep{wainwright2006log, affandi2014learning, cortinovis2021randomized} and many more~\citep{poulin2003testing,chowdhury2021computing,bravyi2021complexity}. 
When the matrix $\A$ is explicitly accessible, retrieving the diagonal entries and computing the trace are straightforward operations. However, in many applications due to computational challenges, the matrix $\A$ cannot be explicitly constructed and can be assessed through only matrix-vector multiplication queries. The common applications of this scenario are where  $\A$ is a transformation of some other matrix $\B$. For example, consider $\B \in \R^{d\times d}$ to be the adjacency matrix of a graph, then $\tr(\B^3)$ is equal to six times the number of triangles in the graph~\citep{avron2010counting,dharangutte2021dynamic}. Further, computing $\A = \B^3$ explicitly to compute the trace requires $O(d^3)$ time, whereas the matrix vector multiplication  $\A\x = \B (\B(\B\x))$ only requires $O(d^2)$ time~\citep{dharangutte2021dynamic}. 

 In this framework, the estimation of the diagonal entries and trace of matrix $\A$ is called implicit or matrix-free diagonal and trace estimations, respectively.
 Moreover,   the exact values of trace and diagonal entries of a matrix $\A \in \R^{d\times d}$ can be computed by performing   $d$ matrix-vector queries
 \begin{align}
     \trn{\A} &= \sum_{i=1}^d  \e_ i^T  \A \e_i, \qquad \text{~and~} \qquad \diagn{\A} = \sum_{i=1}^d \e_i * \A \e_i
 \end{align}    
  where $\e_i$ is $i$-th standard basis vectors of $\R^d$  and $*$ denotes element-wise multiplication. Needless to say this computation is time-consuming when $d$ is large. 
The seminal work of~\cite{hutchinson1989stochastic} addresses this problem by introducing an implicit unbiased estimator for the trace using \rademacher~random variables (Definition~\ref{def:rademacher}). 

\begin{theorem}[\textbf{Hutchinson Trace Estimator~\citep{hutchinson1989stochastic}}]  \label{thm:hutch_trace}
Let $\A \in \R^{d \times d}$ be a symmetric matrix and $\g \in \R^d$ be a random vector whose entries are \iidrade. Then $T:= \g^T \A \g$ is an unbiased estimator of $\trn{\A}$, \textit{i.e.}, $\En{T} = \trn{\A}$ with $\Varn{T} = 2 \left(\|\A\|_{F}^2 - \sum_{j=1}^d a_{j, j}^2 \right)$.
\end{theorem}
Later, \cite{bekas2007estimator} extended Hutchinson's method to the diagonal estimation of a $d \times d$ matrix $\A$ by considering the Hadamard product (Definition~\ref{def:hadamard_prod}) of $\g$ and $\A \g$, \ie, $\g * \A\g$.

\begin{theorem}[\textbf{Diagonal entries estimator of a matrix~\citep{bekas2007estimator}}]  \label{thm:hutch_diag}
 Let $\A \in \R^{d\times d}$ be a matrix, and $\g \in \R^d$ be a random vector whose entries are \iidrade. Then each entry of the vector $\vectornotation{d} := \g * \A \g \in \R^d$ is an unbiased estimator of the diagonal entries of $\A$, \textit{i.e.}, for $i\in [d]$    $\En{d_i} =  a_{i,i}$ with variance $\Varn{d_i} =  \|\a_{i}\|^2 - a_{i,i} $ where $\a_{i}$ denotes the $i$-th row of the matrix $\A$.
\end{theorem}

In modern applications, tensor data structures are prevalent and widely used in many fields, such as graph theory~\citep{shao2015some, farid2022tensor,zhou2023estrada}, quantum computing~\citep{nielsen2001quantum,huang2020predicting}, machine learning~\citep{tao2005supervised,rabanser2017introduction}, signal processing~\citep{sidiropoulos2017tensor}, neuroscience~\citep{mori2006principles,liao2013depression}, computer vision~\citep{aja2009tensors,panagakis2021tensor}.
Further, analogous to the matrix scenario mentioned above, the implicit trace and diagonal estimation can naturally be considered for tensors where the task is to estimate these quantities by minimizing the number of tensor-vector product queries. We can compute the exact values of the trace and diagonal entries of a tensor $\TA \in \R^{d\times \cdots \times d}$ by performing $d$ tensor vector product queries as follows
\begin{align}
\trn{\TA} &= \sum_{i=1}^d  \e_ i^T  \left(\TA  \vtimes_1 \e_i \vtimes_2 \e_i \vtimes_3 \cdots \vtimes_{N-1} \e_i \right), \text{~and~}\\
\diagn{\TA} &= \sum_{i=1}^d \e_i * \left(\TA  \vtimes_1 \e_i \vtimes_2 \e_i \vtimes_3 \cdots \vtimes_{N-1} \e_i \right),
\end{align}
where $\e_i$ is $i$-th standard basis vector of $\R^d$. However, to the best of our knowledge, no analogous extension to Hutchinson's results of matrices~\citep{hutchinson1989stochastic,bekas2007estimator} are known for higher-order tensors that estimate the trace and diagonal entries. We address this problem and give unbiased estimators for the diagonal entries and trace of tensor under the tensor-vector query framework. We summarize our key contributions as follows:\\

   \noindent $\bullet$ \textbf{Contribution 1:} Our first contribution is to propose an unbiased estimator to approximate the diagonal elements of a tensor using the tensor-vector multiplication queries. We define it as follows:
    \begin{definition} \label{def:diag_rad}    
    Let $\TA \in \R^{d \times \cdots \times d}$ be \anordertensor~with each \modeororder~size $d$. Let $\g^{(n)} \in \R^d ~\forall~ n \in [N-1]$ be  random vectors whose entries are \textit{i.i.d.} random variables with mean zero and unit variance. Let $\g := \g^{(1)} * \g^{(2)} * \cdots * \g^{(N-1)}$ where $*$ denotes the element wise product (Definition~\ref{def:hadamard_prod}). Let $\vtimes_n$ denote the mode-$n$ tensor vector multiplication for $n \in [N-1]$. Then, 
    \begin{align}
       \y := \g * \left(\TA \vtimes_1 \g^{(1)} \vtimes_2 \g^{(2)} \vtimes_3 \cdots \vtimes_{N-1} \g^{(N-1)} \right) \label{eq:eqn_est_diag_Rad_intro}
    \end{align}
    gives an estimate of the diagonal entries of $\TA$.
    \end{definition}
   
    In Theorem~\ref{thm:diag_rad_new}, we show that our proposal (Equation~\eqref{eq:eqn_est_diag_Rad_intro}) is an unbiased estimator of the diagonal entries of $\TA$ and provide its variance bound. Further, in Corollaries~\ref{rade_diag_cor} and \ref{Normal_diag_cor}, we provide the concentration bounds on the sample size needed for our estimator to achieve a desired $(\epsilon, \delta)$-approximation of the tensor's diagonal entries when the entries of $\g^{(n)}$  are \textit{i.i.d.} samples from the Rademacher and $\mathcal N(0,1)$ distribution, where  $n \in [N-1]$.\\

\noindent $\bullet$ \textbf{Contribution 2:} We propose an unbiased estimator for computing the tensor trace using the tensor-vector multiplication queries. We define our proposal as follows.
 \begin{definition} \label{def:tr_Rad_tr}
    Let $\TA \in \R^{d \times \cdots \times d}$ be \anordertensor~with each \modeororder~size $d$. Let $\g^{(n)} \in \R^d ~\forall~ n \in [N-1]$ be  random vectors whose entries are \textit{i.i.d.} random variables having mean zero and unit variance. Let $\g := \g^{(1)} * \g^{(2)} * \cdots * \g^{(N-1)}$. Then, 
    \begin{align}
        X:= \g^T \left(\TA \vtimes_1 \g^{(1)} \vtimes_2 \g^{(2)} \vtimes_3 \cdots \vtimes_{N-1} \g^{(N-1)} \right) \label{eq:eqn_est_tr_rad_intro}
    \end{align}
    gives an estimate of the trace of tensor $\TA$.
    \end{definition}

    In Theorem~\ref{thm:trace_estimator_new}, we show that our proposal (Equation~\eqref{eq:eqn_est_tr_rad_intro}) is an unbiased estimator of the trace of the tensor $\TA$ and provide its variance bound.  In Corollaries~\ref{cor:trace_rademacher_new} and \ref{cor:trace_normal_new}, we give the concentration bound on the number of samples required for our proposal to provide an $(\epsilon, \delta)$ - approximation of the tensor trace when the entries of $\g^{(n)}$ are \textit{i.i.d.} samples from Rademacher and $\mathcal N(0,1)$ respectively.\\
    
Note that in this work, we refer to the sum of the diagonal entries of the tensor as a trace of the tensor (Equation~\ref {eq:eqn_tensor_trace}).
Our proposals stated in Definition~\ref{def:diag_rad} and \ref{def:tr_Rad_tr} can be seen as the generalization of diagonal entries estimation of matrices (Theorem~\ref{thm:hutch_diag}) and  Hutchinson's trace estimator (Theorem~\ref{thm:hutch_trace}) to higher order tensors, and simplifies to these estimators when $N=2$ and elements of $\g^{(n)}$ for $n \in [N-1]$ are \textit{i.i.d.} samples from Rademacher distribution. \\

\noindent \textbf{Organization of the paper:} In  Section~\ref{sec:related_work}, we discuss the related work. Section~\ref{sec:preliminaries}, summarises the notations and necessary concepts used in the paper. Section~\ref{sec:analysis}, presents our trace and diagonal entries estimator proposals for tensors with their theoretical analysis.
Section~\ref{sec:experiment}, complements our theoretical analysis via supporting experiments on synthetic datasets. Finally, in Section~\ref{sec:conclusion}, we conclude the discussion followed by some potential open questions of the work.

\section{Related Work} \label{sec:related_work}
The seminal work of~\cite{hutchinson1989stochastic} gives a randomized algorithm called Hutchinson estimator to approximate the trace of a given matrix via matrix-vector multiplication queries. Hutchinson's estimator is based on the observation that for a given matrix $\A\in \R^{d\times d}$, $E[\g^T \A \g] = \tr(\A)$, where $\g  \in \R^d $ whose entries are \textit{i.i.d.} random variables with mean $0$ and variance $1$, or \iidrade~(Theorem~\ref{thm:hutch_trace}). \cite{wong2004computing} suggested that vectors $\g$ can also be taken from the columns of a Hadamard matrix.
\cite{iitaka2004random} generalized Hutchinson's estimator by using random phase vectors with unit magnitude, showing that the resulting estimator has reduced variance compared to Hutchinson's but with increased computational complexity.
Later, \cite{bekas2007estimator} extended the Hutchinson estimator to approximate the diagonal entries of the matrices using matrix-vector queries (Theorem~\ref{thm:hutch_diag}).

The work of ~\cite{avron2011randomized} was the first to give bounds on the number of samples required by the Hutchinson estimator for positive semidefinite matrices to achieve $(\epsilon, \delta)$~approximation (Definition~\ref{def:eps_delta_approx}). Later, \cite{roosta2015improved,skorski2021modern} and \cite{cortinovis2021randomized} also analysed Hutchinson's trace estimator and presented a slightly tighter sample bound compared to \cite{avron2011randomized}. In the context of diagonal estimation, the work of \cite{baston2022stochastic, hallman2023monte} and \cite{dharangutte2023tight} analysed the Hutchinson's diagonal estimator due to \cite{bekas2007estimator} and gave improved concentration bounds to achieve $(\epsilon, \delta)$~approximation. 

 Recently, numerous studies have been proposed to give improved variants of  Hutchinson's estimator. 
\cite{adams2018estimating} applied the control variate method to reduce the variance of Hutchinson's estimator. 
\cite{gambhir2017deflation} and \cite{meyer2021hutch++} used a decomposition approach involving the projection of $\A$ on some matrix $\mathbf{Q}$ which spans $\A$'s top eigenspace to reduce the variance of the Hutchinson estimator.  \cite{meyer2021hutch++} proposed Hutch++ algorithm to estimate the trace of a matrix, which improves the query complexity bound of Hutchinson's trace estimator from $1/\epsilon^2$ matrix vector queries to $ 1/ \epsilon$ matrix vector queries to achieve $(\epsilon, \delta)$ approximation.
Similar to Hutch++, for the diagonal elements estimation problem, \cite{baston2022stochastic} suggested the Diag++ algorithm that achieves a similar improvement in query complexity over \cite{bekas2007estimator}. 
\cite{persson2022improved}  proposed the Nystrom++ algorithm, an improved version of Hutch++ that uses the Nystrom approximation and only requires one pass over the matrix compared to two passes by  Hutch++. \cite{epperly2024xtrace} recently suggested two new methods, \texttt{XTrace} and \texttt{XNysTrace}, which exploited the variance
reduction and the exchangeability principle. These methods achieve errors that are orders of magnitude smaller than  Hutch++. The work of \citep{bujanovic2021norm,meyer2023hutchinson} extended Hutchinson's trace estimator to the Kronecker-matrix-vector oracle model and provided its theoretical analysis.   
 
Recently, the matrix-vector query estimation techniques for trace and diagonal estimation have gained a lot of attention due to their widespread applicability across applications in computational science~\citep{metivier2014full,aster2018parameter}, machine learning~\citep{molchanov2016pruning, eriksson2018scaling}, and optimization~\citep{dauphin2015equilibrated,yao2021adahessian}. To the best of our knowledge, the implicit trace and diagonal estimation methods for tensors have not been studied. 
This work considers this problem and initiates its study.   

\section{Preliminaries} \label{sec:preliminaries}
We use $[d]$ to denote the set $\{1,\ldots,d\}$. We denote tensors by capital calligraphic letters, matrices by upper boldface letters, vectors by lower boldface letters, scalars by normal lowercase letters, and random variables by italics. $\TA \in \R^{d_1 \times \cdots \times d_N}$ denotes \anordertensor~having each \modeororder~size $d_i$ for $i \in [N]$ and we represent its $(i_1, \ldots, i_N)$-th  element  by $a_{i_1, \ldots, i_N}$.  
The order of a tensor is the number of dimensions and is also known as ways or modes. We interchangeably use the terms mode and order to denote the number of dimensions of a tensor.
The Frobenius norm of a general tensor $ \TA \in \R^{d_1 \times \cdots \times d_N}$ is denoted by $ \| \TA\|_{F} := \left(\sum_{(i_1,\ldots, i_N) } a_{i_1, \ldots, i_N}^2\right)^{1/2}$. 
$\M \in \R^{m \times n}$ represents a $m \times n$ matrix and $m_{i,j}$ denotes its $(i,j)$-th element. $\a \in \R^{d}$ denotes a $d$-dimensional vector and $a_i$ represents its $i$-th element.

\begin{definition}[\textbf{Hadamard Product}~\citep{kolda2009tensor}] \label{def:hadamard_prod}
Let $\A, \B \in \R^{m \times n }$. Their Hadamard product  $\A * \B \in \R^{m \times n}$  is defined as follows: $(\A * \B)_{i,j} = a_{i,j} b_{i,j}$ for $i \in [m]$ and $j \in [n]$.
For vectors $\a, \b \in \R^{n}$, $\a * \b = [a_1 b_1, \cdots, a_n b_n] \in \R^{n}$.
\end{definition}

\begin{definition}[\textbf{Diagonal Elements and trace of a Tensor~\citep{kolda2009tensor,qi2017tensor}}] Let $\TA \in \R^{d \times \cdots \times d}$ be \anordertensor~with each \modeororder~size $d$. Then, any element $a_{i_1, \ldots, i_N}$ is a diagonal element of $\TA$ iff $i_1 = i_2 = \ldots = i_N$. Further,  we denote  $\trn{\TA}$ as the trace of tensor $\TA$ and define it as the sum of the diagonal entries ~\citep[Page 22]{qi2017tensor}, 
\begin{align}
    \trn{\TA} = \sum_{i=1}^d a_{i, \ldots, i}. \numberthis \label{eq:eqn_tensor_trace}
\end{align}
\end{definition}

\begin{definition}[\textbf{Mode-$n$ Tensor Vector Multiplication~\citep{kolda2009tensor}}]
    Let $\TA \in \R^{d_1 \times \cdots \times d_N}$ be \anordertensor~and $\x \in \R^{d_n}$ be a vector. Then mode-$n$ tensor vector multiplication is denoted as $\TA \vtimes_n \x \in \R^{d_1\times \cdots \times d_{n-1} \times d_{n+1} \times \cdots \times d_N}$ and defined elementwise as follows:
    \begin{align}
         \left(\TA \vtimes_n \x \right)_{i_1, \ldots, i_{n-1}, i_{n+1}, \ldots, i_N} &= \sum_{i_n =1}^{d_n} a_{i_1, \cdots, i_N} x_{i_n} \nonumber.
    \end{align}
\end{definition}

\begin{definition}[\textbf{\Rademacher~Distribution}] \label{def:rademacher}
A random variable $X$ comes from a \Rademacher~distribution if $X$ takes on  values $\{+1,-1\}$ each with probability $1/2$. We use the term \textit{i.i.d.} Rademacher to denote random variables \textit{i.i.d.} from the \Rademacher~distribution.
\end{definition}

\begin{definition}[\textbf{$(\epsilon, \delta)$ -Approximator}]\label{def:eps_delta_approx}
    A randomized estimator $X$ is said to be a $(\epsilon, \delta)$-approximator of quantity $\xi$ if $ \Prn{|X - \xi| \leq  \epsilon \cdot \xi} \geq 1 - \delta.$
\end{definition}
The following theorem states Hypercontractivity concentration inequality, which is an extension of the Hanson-Wright inequality, and its proof can be found in \citep{schudy2012concentration}. We use it to derive the number of samples required by our proposals (Definitions~\ref{def:diag_rad} and \ref{def:tr_Rad_tr}) to achieve $(\epsilon, \delta)$-approximation.
\begin{theorem}[\textbf{Hypercontractivity Concentration Inequality}~\citep{schudy2012concentration}] \label{thm:hpercontractractivity}
    Consider a degree $q$ polynomial $f(Y) = f(Y_1, \ldots, Y_n)$ of independent centered Gaussian or Rademacher random variables $Y_1, \ldots,Y_n$. Then
    \begin{align}
        \mathrm{Pr}\left[|f(Y) - \E\left[f(Y)\right]| \geq \lambda \right] \leq e^2 \cdot e^{-\left(\frac{\lambda^2}{R \cdot \Var(f(Y))} \right)^{1/q}}
    \end{align}
    where $\Var(f(Y))$ is the variance of the random variable $f(Y)$ and $R>0$ is an absolute constant.    
\end{theorem}
 
If the estimator is not required to be linear, an alternative way to achieve an $(\epsilon, \delta)$-approximation is the \textit{median-of-means} trick. The following lemma states the result for the median-of-means estimator.
\begin{lem}[\cite{nemirovskij1983problem,lugosimean}] \label{lem:median_of_means}
Let $Y_1, \ldots, Y_{rs}$ be \textit{i.i.d.} random variables with mean $\mu$ and variance $\sigma^2$. Divide the samples into $r$ disjoint groups, each of size $s$, and compute the empirical mean of each group. Let the median-of-means estimator be defined as
\[
\mu_{MM} := \operatorname{median} \left( \frac{1}{s} \sum_{t=1}^{s} Y_t,\ \frac{1}{s} \sum_{t=s+1}^{2s} Y_t,\ \ldots,\ \frac{1}{s} \sum_{t=(r-1)s+1}^{rs} Y_t \right).
\]
Then, for any $\delta \in (0, 1)$, if $r = 8 \log(1/\delta)$, the following is true with probability at least $1 - \delta$ 
\[
|\mu_{MM} - \mu| \leq \sigma \sqrt{ \frac{4}{s} }.
\]
\end{lem}

\section{Our estimators and their analysis} \label{sec:analysis}

\subsection{Intuition for our estimators }

{\color{black}
Recall that for a square matrix $\A$ of size $d$,  the Hutchinson trace estimator is defined as $\g^T \A \g $ where $\g \in \R^d$ is a random vector whose entries are \textit{i.i.d} Rademacher random variables. In the case of matrices, each row/column of the matrix contains exactly one diagonal element. The idea of the Hutchinson estimator is to compress each row of the matrix in a one-dimensional summary and then, from these summaries, estimate the corresponding diagonal elements and compute their sum. The Hutchinson trace estimator operation can be considered into two parts: assume that $\v:= \A \g$, $\v \in \R^d$ and the $i$-th element of $\v$ represents the one-dimensional summary of the $i-$th row of the matrix $\A$. Note that $\g^T \v$ is the Hutchinson trace estimator,   and is $\g^T \v=\sum_{i=1}^d g_i v_i$ where $g_i v_i = a_{i,i} \cdot g_{ii}^2 + \sum_{j \in [d] \setminus \{i\}} a_{i,j} \cdot g_j \cdot g_i$. It is easy to verify that $\En{g_i v_i}=a_{i,i} $, and therefore gives an unbiased estimator of $a_{i,i}$. Consequently, $\En{g^T \v}=\sum_{i=1}^d a_{i,i}$. The same idea extends to the Hutchinson diagonal estimator, where instead of computing the sum of the recovered diagonal entries, they are returned in vector form by leveraging the Hadamard product, \ie, $\g^T * \A\g$.

 Our proposals extend the above idea to a higher-order tensor.  In the case of a higher-order tensor, each slice (data subset obtained by fixing one index and letting others free) of the tensor consists of exactly one diagonal element.  Our proposed trace estimator compresses each slice into a one-dimensional summary, recovers the corresponding diagonal element from them, and computes their sum. Let's understand the working of our trace estimator using a $N$-order tensor $\TA \in \R^{d \times \cdots \times d}$. Let $\g^{(1)},\ldots, \g^{(N-1)}$ be $d$-dimensional random vectors whose entries are \textit{i.i.d.} Rademacher and $\g:= \g^{(1)} * \cdots *\g^{(N-1)} $ where $*$ denotes the Hadamard product.  In our proposed trace estimator (Definition \ref{def:tr_Rad_tr}), the operation $\TA \vtimes \g^{(1)} \vtimes \cdots \vtimes \g^{(N-1)}$ results in a $d$-dimension vector, whose $i$-th element presents the one-dimensional summary of the $i$-th slice of the tensor $\TA$ and the operation $\g^T \left(\TA \vtimes \g^{(1)} \vtimes \cdots \vtimes \g^{(N-1)}\right)$  recovers the corresponding diagonal elements from their one-dimensional summaries of the slices and returns their sum. Similarly, this idea leads to diagonal estimation if we compute the Hadamard product of $\g$ with $\left(\TA \vtimes \g^{(1)} \vtimes \cdots \vtimes \g^{(N-1)}\right)$ (Definition~\ref{def:diag_rad}).

}

In this section, we define our estimators for the diagonal entries and trace of the tensor.  Following the definition, we provide a theoretical analysis of our proposals by showing that our estimates are unbiased. Then, we provide bounds on their variance, followed by a concentration analysis.

\subsection{Diagonal entries estimator}

\label{section_diagonal_estimator}
In the following theorem, we give an unbiased estimator to estimate each diagonal element of a tensor $\TA$ and provide a bound on its variance.

\begin{restatable}{theorem}{diagRadnew}~\label{thm:diag_rad_new}
   Let $\TA \in \R^{d \times \cdots \times d}$ be \anordertensor~with each \modeororder~size $d$.     
    Let $\g^{(n)} \in \R^d$ for $n \in [N-1]$ be random vectors where entries are mean zero, have a unit second moment and finite fourth moment, and are pairwise independent, i.e. $\Expt{g^{(n)}_i} = 0, \Expt{\left(g^{(n)}_i\right)^2} = 1, \Expt{\left(g^{(n)}_i\right)^4} < \infty$, $\Expt{g^{(n)}_i g^{(m)}_j} = \Expt{g^{(n)}_i}\Expt{g^{(m)}_j} \forall~m \neq n$ or $i \neq j$. Let $\g := \g^{(1)} * \g^{(2)} * \cdots * \g^{(N-1)}$. Then, each entry of 
    \begin{align}
       \y := \g * \left(\TA \vtimes_1 \g^{(1)} \vtimes_2 \g^{(2)} \vtimes_3 \cdots \vtimes_{N-1} \g^{(N-1)} \right) \label{eq:eqn_est_diag_Rad}
    \end{align}
    gives an unbiased estimate of the diagonal elements of tensor $\TA$, \textit{i.e.} for $i\in[d]$, $\En{y_i}  = a_{i, \ldots, i}$
    with variance
    \begin{align}
        \Varn{y_i} &=\sum_{s=0}^{N-1} \Expt{\RVar^4}^s \left(  \sum_{\scriptsize
\begin{array}{c}
 (j_1, \ldots, j_{N-1})\\
\text{where $s$ of}~j_t, t \in [N-1]\\
\text{are equal to $i$}
\end{array}}    a_{j_1, \ldots, j_{N-1},i}^2   \right)  - a_{i, \ldots, i}^2, \numberthis \label{eq:310324_1}
    \end{align}
 where $\RVar$ is a random variable identically distributed to $g_i^{(n)}$'s, where  $i\in [d]$ and $n \in [N-1]$. For $p \neq q$, covariance
{\color{black}
\begin{align}
    \Cov\left(y_p,y_q\right) = \sum_{\scriptsize\begin{array}{c}
    (j_1,\hdots,j_{N-1}) \in \{p,q\} \\
    \text{and}~(k_1,\hdots,k_{N-1}) \in \{p,q\} \\
    \text{and}~j_t \neq k_t ~\forall~ t \in [N-1]
    \end{array}} a_{j_1,\hdots,j_{N-1},p}a_{k_1,\hdots,k_{N-1},q} - a_{p,\hdots,p}a_{q,\hdots,q}. \label{eq:310324_2}
\end{align}}
\end{restatable}

\begin{proof}
    From Equation~\eqref{eq:eqn_est_diag_Rad}, we have
    \begin{align*}
        \y &= \begin{bmatrix}
            y_1\\
            \vdots\\
            y_d
        \end{bmatrix} = \begin{bmatrix}
            g_1 \displaystyle \sum_{j_1, \ldots, j_{N-1}} a_{j_1, \ldots, j_{N-1},1}  \prod_{t=1}^{N-1}g_{j_t}^{(t)}\\
            \vdots \\
            g_d \displaystyle \sum_{j_1, \ldots, j_{N-1}} a_{j_1, \ldots, j_{N-1},d} \prod_{t=1}^{N-1}g_{j_t}^{(t)}\\
        \end{bmatrix}\\
        & = \begin{bmatrix}
        \left(\displaystyle\prod_{s=1}^{N-1} g_1^{(s)}\right) \displaystyle \sum_{j_1, \ldots, j_{N-1}} a_{j_1, \ldots, j_{N-1},1} \prod_{t=1}^{N-1}g_{j_t}^{(t)}\\
            \vdots\\
            \left(\displaystyle\prod_{s=1}^{N-1} g_d^{(s)}\right) \displaystyle \sum_{j_1, \ldots, j_{N-1}} a_{j_1, \ldots, j_{N-1},d} \prod_{t=1}^{N-1}g_{j_t}^{(t)}
        \end{bmatrix}\\
        &= \begin{bmatrix}
            \displaystyle \sum_{j_1, \ldots, j_{N-1}} a_{j_1, \ldots, j_{N-1},1}  \prod_{t=1}^{N-1} g_{j_t}^{(t)}g_1^{(t)}\\
            \vdots\\
            \displaystyle \sum_{j_1, \ldots, j_{N-1}} a_{j_1, \ldots, j_{N-1},d}  \prod_{t=1}^{N-1} g_{j_t}^{(t)}g_d^{(t)} 
        \end{bmatrix}.
    \end{align*}
We first compute the expected value of $y_i$ for $i \in [d]$.
    \begin{align}
        \En{y_i} &=   \sum_{j_1=1}^d \sum_{j_2=1}^d \cdots \sum_{j_{N-1}=1}^d  a_{j_1,\ldots, j_{N-1}, i} \, \prod_{t=1}^{N-1} \En{g_{j_t}^{(t)}g_i^{(t)}}  =  a_{i, \cdots, i}\label{eq:eqn_unb_Est_diag_rad_proof} 
    \end{align}
    where in Equation \eqref{eq:eqn_unb_Est_diag_rad_proof}, we use the following fact
    \begin{align}
        \En{g^{(n)}_{j_{n}} g_{i}^{(n)}} = \begin{cases}
            1, & \text{if } j_n = i\\
            0, & \text{otherwise} 
        \end{cases} \quad \forall ~~ n \in [N-1] \text{ and } j_n, i \in [d], \nonumber
    \end{align}
to deduce that the only non-zero product term occurs when $j_1 =  \hdots = j_{N-1} = i$, leading  to the diagonal element $a_{i,\ldots,i}$. We next compute the variance of $y_i$ for $i \in [d]$.
\begin{align}
        &\Varn{y_i}   = \En{y_i^2} - \En{y_i}^2 \nonumber\\
\begin{split}
        &= \E \Bigg[\sum_{j_1, \ldots, j_{N-1}} a_{j_1, \ldots, j_{N-1},i}^2 \prod_{t=1}^{N-1} (g_{j_t}^{(t)})^2(g_{i}^{(t)})^2  \\
       & \phantom{=} + \qquad \qquad \sum_{\mathclap{(j_1,\ldots, j_{N-1}) \neq (k_1,\ldots, k_{N-1})}} \qquad a_{j_1, \ldots, j_{N-1},i} a_{k_1, \ldots, k_{N-1},i} \prod_{t=1}^{N-1} g_{j_t}^{(t)}g_{k_t}^{(t)}(g_{i}^{(t)})^2 \Bigg]  - a_{i, \ldots, i}^2
\end{split} \nonumber \\
\begin{split}
        &= \sum_{j_1, \ldots, j_{N-1}} a_{j_1, \ldots, j_{N-1},i}^2 \prod_{t=1}^{N-1} \En{(g_{j_t}^{(t)})^2(g_{i}^{(t)})^2}  \\
        & \phantom{=} + \qquad \qquad \sum_{\mathclap{(j_1,\ldots, j_{N-1}) \neq (k_1,\ldots, k_{N-1})}} \qquad a_{j_1, \ldots, j_{N-1},i} a_{k_1, \ldots, k_{N-1},i} \prod_{t=1}^{N-1} \En{ g_{j_t}^{(t)}g_{k_t}^{(t)}(g_{i}^{(t)})^2} - a_{i, \ldots, i}^2
\end{split} \label{eq:eqn_unb_Est_var_diag_rad_proof_prod} \\
 & = \sum_{s=0}^{N-1} \Expt{\RVar^4}^s \left(  \sum_{\scriptsize
\begin{array}{c}
 (j_1, \ldots, j_{N-1})\\
\text{where $s$ of}~j_t, t \in [N-1]\\
\text{are equal to $i$}
\end{array}}    a_{j_1, \ldots, j_{N-1},i}^2   \right)  + 0 - a_{i, \ldots, i}^2.  \label{eq:eqn_var_yi_Rad_special_struct_new}
    \end{align}
In Equation \eqref{eq:eqn_var_yi_Rad_special_struct_new}, $\RVar$ denotes a random variable which is identical to $g_i^{(n)}$'s for $i \in [d]$ and $n \in [N-1]$. The Equation \eqref{eq:eqn_var_yi_Rad_special_struct_new} holds due to the following fact
    \begin{align}
        \E \left[ (g^{(t)}_{j_{t}})^2 (g_{i}^{(t)})^2\right] = \begin{cases}
            \Expt{\RVar^4}, & \text{if } j_t = i\\
            \Expt{\RVar^2}\Expt{\RVar^2} = 1, & \text{otherwise} 
        \end{cases} ~ \forall ~ t \in [N-1] \text{ and } j_t, i \in [d]\nonumber
    \end{align}
and hence the product $\prod_{t=1}^{N-1} \En{(g_{j_t}^{(t)})^2(g_{i}^{(t)})^2} = \Expt{\RVar^4}^{s}$, where $s$ is the number of times $j_t = i$, while $\En{ g_{j_n}^{(n)}g_{k_n}^{(n)}(g_{i}^{(n)})^2} = 0~\forall ~~ n \in [N-1] \text{ and } j_n, i \in [d]$.
Finally, we compute the covariance $\Cov\left(y_p,y_q\right)$ for $p, q \in[d], p \neq q$ by first computing $\Expt{y_py_q}$.
\begin{align}
\Expt{y_py_q} &   = \mathbb E\left[ \sum_{j_1,\hdots,j_{N-1}}a_{j_1,\hdots,j_{N-1},p}a_{j_1,\hdots,j_{N-1},q} \prod_{t=1}^{N-1}(g_p^{(t)})(g_q^{(t)}) (g_{j_t}^{(t)}) (g_{j_t}^{(t)})    \right. \notag \\
 & \phantom{=}+ \left. \sum_{(j_1,\hdots,j_{N-1}) \neq (k_1,\hdots,k_{N-1})} a_{j_1,\hdots,j_{N-1},p}a_{k_1,\hdots,k_{N-1},q} \prod_{t=1}^{N-1} g_{j_t}^{(t)} g_{k_t}^{(t)}(g_p^{(t)})(g_q^{(t)}) \right] \notag \\
  & =  \sum_{j_1,\hdots,j_{N-1}}a_{j_1,\hdots,j_{N-1},p}a_{j_1,\hdots,j_{N-1},q} \prod_{t=1}^{N-1}\Expt{(g_p^{(t)})(g_q^{(t)}) (g_{j_t}^{(t)}) (g_{j_t}^{(t)})}     \notag \\
 & + \sum_{(j_1,\hdots,j_{N-1}) \neq (k_1,\hdots,k_{N-1})} a_{j_1,\hdots,j_{N-1},p}a_{k_1,\hdots,k_{N-1},q} \prod_{t=1}^{N-1} \Expt{g_{j_t}^{(t)} g_{k_t}^{(t)}(g_p^{(t)})(g_q^{(t)})} \label{covariance_derivation_1}\\
 & = \sum_{(j_1,\hdots,j_{N-1}) \neq (k_1,\hdots,k_{N-1})} a_{j_1,\hdots,j_{N-1},p}a_{k_1,\hdots,k_{N-1},q} \prod_{t=1}^{N-1} \Expt{g_{j_t}^{(t)} g_{k_t}^{(t)}(g_p^{(t)})(g_q^{(t)})} \label{covariance_derivation_2} \\
 & =  \sum_{\scriptsize\begin{array}{c}
    (j_1,\hdots,j_{N-1}) \in \{p,q\} \\
    \text{and}~(k_1,\hdots,k_{N-1}) \in \{p,q\} \\
    \text{and}~j_t \neq k_t ~ \forall ~ t \in [N-1]
    \end{array}} a_{j_1,\hdots,j_{N-1},p}a_{k_1,\hdots,k_{N-1},q} \label{covariance_derivation_3}
\end{align}
In our derivation, Equation~\eqref{covariance_derivation_2} follows from Equation~\eqref{covariance_derivation_1} since
\begin{align}
\Expt{(g_p^{(t)})(g_q^{(t)}) (g_{j_t}^{(t)}) (g_{j_t}^{(t)})} & = \left\{ \begin{array}{r l}
\Expt{g_p^{(t)}}\Expt{g_q^{(t)}}\Expt{(g_{j_t}^{(t)})^2} = 0, & \text{if}~j_t \neq p\\
& ~~ \text{and } j_t \neq q \\
\Expt{g_p^{(t)}} \Expt{(g_j^{(t)})^3} = 0, & \text{if}~j_t = q \\
\Expt{g_q^{(t)}} \Expt{(g_j^{(t)})^3} = 0, & \text{if}~j_t = p
\end{array}\right.
\end{align}
for a fixed $t \in [N-1]$ hence the product of expectations in the first summation is zero. We next observe that
\begin{align}
\Expt{g_{j_t}^{(t)}g_{k_t}^{(t)}g_p^{(t)}g_q^{(t)}} = \left\{ \begin{array}{r l}
\Expt{(g_p^{(t)})^2(g_q^{(t)})^2} = 1, & \text{if}~j_t=p, k_t=  q, ~\text{or}~j_t=q,k_t = p\\
0, & \text{otherwise}
\end{array}\right.
\end{align}
for a fixed $t \in [N-1]$. For the product of expectations to be non-zero, we need to look at the summation of terms $a_{j_1,\hdots,j_{N-1},p}a_{k_1,\hdots,k_{N-1},q}$ simultaneously fulfilling the conditions $(j_1,\hdots,j_{N-1})\in \{p,q\}, (k_1,\hdots,k_{N-1})\in \{p,q\}$ and $j_t \neq k_t$ for all $t \in [N-1]$, therefore Equation~\eqref{covariance_derivation_3} follows from Equation~\eqref{covariance_derivation_2}. It thus follows that
\begin{align*}
\Cov\left(y_p,y_q\right) & = \Expt{y_py_q} - \Expt{y_p}\Expt{y_q} \\
 & = \sum_{\scriptsize\begin{array}{c}
    (j_1,\hdots,j_{N-1}) \in \{p,q\} \\
    \text{and}~(k_1,\hdots,k_{N-1}) \in \{p,q\} \\
    \text{and}~j_t \neq k_t ~\forall~t \in[N-1]
    \end{array}} a_{j_1,\hdots,j_{N-1},p}a_{k_1,\hdots,k_{N-1},q}  - a_{p,\hdots,p}a_{q,\hdots,q} \numberthis \label{eq:eqn_covar_yi_Rad_special_struct_new}.
\end{align*}
Equations~\eqref{eq:eqn_unb_Est_diag_rad_proof} \eqref{eq:eqn_var_yi_Rad_special_struct_new}, and \eqref{eq:eqn_covar_yi_Rad_special_struct_new} complete a proof the theorem.
\end{proof}

The following corollaries provide bounds on the variance of the diagonal estimator when elements of random vector $\g^{(n)}$ for $n \in [N-1]$ are \textit{i.i.d.}   Rademacher and Gaussian. They also state the bounds on the number of samples required to be $(\epsilon, \delta)$ estimator.

\begin{restatable}{cor}{diagRadCor}\label{rade_diag_cor}
If the entries of $\g^{(n)}$ for $n \in [N-1]$ in Theorem~\ref{thm:diag_rad_new} are  i.i.d. Rademacher, then 
\begin{align}
        \Varn{y_i} &= \sum_{j_1, \ldots, j_{N-1}} a_{j_1, \ldots, j_{N-1},i}^2   - a_{i, \ldots, i}^2, ~~~~~~~~\forall i \in [d]\numberthis \label{eq:eqn_0410_1}.
    \end{align}
Further, for any diagonal element $a_{i, \ldots, i}$ of $\TA$, the mean of its $K$ \textit{i.i.d.} estimates, where  $$K \geq O\left( \left(\sum_{j_1, \ldots, j_{N-1}}^d a_{j_1, \ldots, j_{N-1},i}^2   - a_{i, \ldots, i}^2 \right) \left(2 +\log(1/\delta) \right)^{2(N-1)}/\left(\epsilon^2 \cdot  a_{i, \ldots, i}^2\right) \right), $$  obtained using different sets of $\g^{(n)}$'s , for $n\in [N-1]$, gives  an $(\epsilon, \delta)$ approximation for $a_{i, \ldots, i}$. 
\end{restatable}
\begin{proof}
From Equation~\eqref{eq:310324_1} of Theorem~\ref{thm:diag_rad_new}, we have
\begin{align}
    \Varn{y_i} &=\sum_{s=0}^{N-1} \Expt{\RVar^4}^s \left(  \sum_{\scriptsize
\begin{array}{c}
 (j_1, \ldots, j_{N-1})\\
\text{where $s$ of}~j_t, t \in [N-1]\\
\text{are equal to $i$}
\end{array}}    a_{j_1, \ldots, j_{N-1},i}^2   \right)  - a_{i, \ldots, i}^2,
\end{align}
where $\RVar$ is a random variable with a distribution identical to the entries of $\g^{(n)}$. The fourth moment of Rademacher distribution is $1$, which implies $\Expt{\RVar^4}=1$. Thus, from the above equation, we have
\begin{align*}
  \Varn{y_i} &=\sum_{s=0}^{N-1}  \left(  \sum_{\scriptsize
\begin{array}{c}
 (j_1, \ldots, j_{N-1})\\
\text{where $s$ of}~j_t, t \in [N-1]\\
\text{are equal to $i$}
\end{array}}   a_{j_1, \ldots, j_{N-1},i}^2   \right)  - a_{i, \ldots, i}^2\\
& = \sum_{j_1, \ldots, j_{N-1}} a_{j_1, \ldots, j_{N-1},i}^2   - a_{i, \ldots, i}^2.
\end{align*}
     Let $Y := \frac{1}{K} \sum_{k=1}^{K} y_i^{(k)}$ where $y_{i}^{(k)}$ for $k \in [K]$ is the estimate of $a_{i, \ldots,i}$ obtained using the $k$-th set of $\g^{(n)}$'s for $n \in [N-1]$. Then
\begin{align*}
   \Varn{Y} & = \Varn{\frac{1}{K} \sum_{k=1}^{K} y_{i}^{(k)}}\\
   &= \frac{1}{K^2} \sum_{k=1}^K \Varn{y_{i}^{(k)}} \quad \left[\because y_{i}^{(k)} ~\text{ for $k \in [K]$ are \textit{i.i.d.} estimates}  \right]\\
    & = \frac{1}{K^2} \sum_{k=1}^K \left(\sum_{j_1, \ldots, j_{N-1}} a_{j_1, \ldots, j_{N-1},i}^2   - a_{i, \ldots, i}^2\right)\\
    &= \frac{\sum_{j_1, \ldots, j_{N-1}} a_{j_1, \ldots, j_{N-1},i}^2   - a_{i, \ldots, i}^2}{K}.
\end{align*}
The variance of $Y$ is bounded, and $Y$ is a polynomial of degree $2(N-1)$ of independent Rademacher random variables (the entries corresponding to distinct sets of $\g^{(n)}$'s for $n \in [N-1]$). Then, for some absolute constant $R$ by utilizing the  Hypercontractivity Concentration Inequality (extension of the Hanson-Wright inequality) stated in Thoerem~\ref{thm:hpercontractractivity}, we have
\begin{align*}
    \Prn{|Y - \E[Y] | \geq \epsilon \cdot a_{i, \ldots, i}} &\leq e^2 \cdot e^{- \left( \frac{\epsilon^2 \cdot a_{i,\ldots,i}^2 \cdot K}{R \cdot  \left(\sum_{j_1, \ldots, j_{N-1}}^d a_{j_1, \ldots, j_{N-1},i}^2- a_{i,\ldots,i}^2 \right)}\right)^{\frac{1}{2(N-1)}}} \numberthis \label{eq:eq280324}\\
    &\leq \delta  \numberthis \label{eq:eqn0410_2}
\end{align*}
    $\text{if we choose } K \geq  \frac{ 2 \cdot R \cdot \left(\sum_{j_1, \ldots, j_{N-1}} a_{j_1, \ldots, j_{N-1},i}^2   - a_{i, \ldots, i}^2 \right) \left(2 + \log(1/\delta) \right)^{2(N-1)}}{\epsilon^2 \cdot a_{i, \ldots,i}^2}$ in Equation~\eqref{eq:eq280324}.
\end{proof}

\begin{restatable}{cor}{diagNorCor}\label{Normal_diag_cor}
If the entries of $\g^{(n)}$ for $n \in [N-1]$  in Theorem~\ref{thm:diag_rad_new} are \textit{i.i.d.} $\mathcal N(0,1)$, then  
\begin{align}
        \Varn{y_i} &=  \sum_{s=0}^{N-1} 3^s \left(  \sum_{\scriptsize
\begin{array}{c}
 (j_1, \ldots, j_{N-1})\\
\text{where $s$ of}~j_t, t \in [N-1]\\
\text{are equal to $i$}
\end{array}}    a_{j_1, \ldots, j_{N-1},i}^2   \right)   - a_{i, \ldots, i}^2  ~~~~~~~~\forall i \in [d]  \label{eq:eqn_var_yi_special_struct_new_normal}\\
 & \leq \left(3^{N-1} -1 \right) \,a_{i, \ldots, i}^2 + 3^{N-2} \sum_{(j_1, \ldots, j_{N-1}) \in [d]^{N-1} \setminus \{(i, \ldots, i)\}} a_{j_1, \ldots, j_{N-1},i}^2. \numberthis \label{eq:eqn_var_yi_Rad}
    \end{align}
Further, for any diagonal element $a_{i, \ldots, i}$ of $\TA$, the average of  its $K$ \textit{i.i.d.} estimates    obtained using different sets of $\g^{(n)}$'s for $n \in [N-1]$, gives an $(\epsilon, \delta)$ approximation of $a_{i, \ldots,i}$ for 

\resizebox{0.97\textwidth}{!}{$K \geq O\left( \hspace{-0.2cm}\left( \left(3^{N-1} -1 \right) \, a_{i, \ldots, i}^2 \, + \, 3^{N-2} \hspace{-0.6cm} \displaystyle \sum_{(j_1, \ldots, j_{N-1}) \in [d]^{N-1} \setminus \{(i, \ldots, i)\}} \hspace{-0.8cm}a_{j_1, \ldots, j_{N-1},i}^2\right) \left( 2 + \log(1/\delta)\right)^{2(N-1)}/\left(\epsilon^2 \cdot a_{i, \ldots, i}^2\right) \right)$}. 

\end{restatable}

\begin{proof}
From Equation~\eqref{eq:310324_1} of Theorem~\ref{thm:diag_rad_new}, we have
\begin{align*}
    \Varn{y_i} &=\sum_{s=0}^{N-1} \Expt{\RVar^4}^s \left(  \sum_{\scriptsize
\begin{array}{c}
 (j_1, \ldots, j_{N-1})\\
\text{where $s$ of}~j_t, t \in [N-1]\\
\text{are equal to $i$}
\end{array}}    a_{j_1, \ldots, j_{N-1},i}^2   \right)  - a_{i, \ldots, i}^2,
\end{align*}
where $\RVar$ is a random variable with a distribution identical to the entries of $\g^{(n)}$, i.e., $\RVar$ follows a standard normal distribution. The fourth moment of $\mathcal{N}(0,1)$ is $3$, which implies $\Expt{\RVar^4}=3$. Thus, by using this fact in the above equation, we have
\begin{align*}
     \Varn{y_i} &=\sum_{s=0}^{N-1}  3^s \left(  \sum_{\scriptsize
\begin{array}{c}
 (j_1, \ldots, j_{N-1})\\
\text{where $s$ of}~j_t, t \in [N-1]\\
\text{are equal to $i$}
\end{array}}    a_{j_1, \ldots, j_{N-1},i}^2   \right)  - a_{i, \ldots, i}^2\\
& = 3^{N-1} a_{i,\ldots,i}^2 + \sum_{s=0}^{N-2}  3^s \left(  \sum_{\scriptsize
\begin{array}{c}
 (j_1, \ldots, j_{N-1})\\
\text{where $s$ of}~j_t, t \in [N-1]\\
\text{are equal to $i$}
\end{array}}    a_{j_1, \ldots, j_{N-1},i}^2   \right)  - a_{i, \ldots, i}^2\\
& \leq \left(3^{N-1} -1 \right) \,a_{i, \ldots, i}^2 + 3^{N-2} \sum_{(j_1, \ldots, j_{N-1}) \in [d]^{N-1} \setminus \{(i, \ldots, i)\}} a_{j_1, \ldots, j_{N-1},i}^2. \numberthis
\end{align*}
We can easily prove the concentration bound by utilizing the  Hypercontractivity Concentration Inequality stated in Theorem~\ref{thm:hpercontractractivity} and employing
the same steps as used in the proof of Corollary~\ref{rade_diag_cor}.
\end{proof}

\noindent \textbf{Note:} Equation \eqref{eq:eqn_var_yi_Rad} provides a crude upper bound for the variance of $y_i$, while Equation \eqref{eq:eqn_var_yi_special_struct_new_normal} is exact. 
In the summation $\displaystyle \sum_{\scriptsize
\begin{array}{c}
 (j_1, \ldots, j_{N-1})\\
\text{where $s$ of}~j_t, t \in [N-1]\\
\text{are equal to $i$}
\end{array}}    a_{j_1, \ldots, j_{N-1},i}^2 $, for each $s$, there are $\left(\begin{array}{c}
N-1 \\
s
\end{array}\right)$ choices $s$ of  $(j_1,\ldots,j_{N-1})$ to equal $i$, and for a fixed configuration of $(j_1,\ldots,j_{N-1})$, there are $(d-1)^{N-1-s}$ possible terms $a_{j_1, \ldots, j_{N-1},i}^2$. 
Given that $d \gg N $ in practical applications, we can expect that the majority of terms in the summation over $ (j_1, \ldots, j_{N-1}) $ will satisfy $ j_t \ne i$ for all $ t \in [N-1]$. In other words, the proportion of terms where none of the indices equal $ i $ approaches $1$ as $d$ increases.
Moreover, if the tensor $\mathcal{A}$ has a special structure, or if most off-diagonal elements in the tensor $\TA$ are approximately equal to each other, say $a_{j_1, \ldots, j_{N-1},i} \approx \tilde{a}$, then we can use Equation \eqref{eq:eqn_var_yi_special_struct_new_normal} to get an approximation for the variance by estimating
\begin{align}
    \Varn{y_i} & \approx \left(3^{N-1} -1 \right) \,a_{i, \ldots, i}^2 + \sum_{s=0}^{N-2} 3^s \left(\begin{array}{c}
N-1 \\
s
\end{array}\right)(d-1)^{N-1-s} \tilde{a}.  
\end{align}

The sample bound in the above corollaries has the exponential dependence on the tensor order $N$ in the term $\log(1/\delta)$. 
If the $(\epsilon, \delta)$ estimator is not required to be linear, we can eliminate it using the \textit{median-of-means} (Lemma~\ref{lem:median_of_means}) trick. The following corollaries provide bounds with improved dependence on $\delta$ by exploiting the results stated in Lemma~\ref{lem:median_of_means}.

\begin{cor}
    Suppose $y_i^{(1)}, \ldots, y_i^{(K)}$ are the \text{i.i.d.} estimates of $a_{i, \ldots, i}$ obtained using $K$ different set of $\g^{(n)}$ for $n \in [N-1]$ in Theorem 5, where entries of $\g^{(n)}$ are \textit{i.i.d} Rademacher. Divide the $K$ estimates randomly into $r$ disjoint groups. The median-of-means of these $r$ groups gives an $(\epsilon,\delta)$ approximation  for $a_{i,\ldots,i}$ for $K \geq \frac{32(\sum_{j_1, \ldots j_{N-1},i} a_{j_1, \ldots,j_{N-1},i}^2 - a_{i,\dots,i}^2) \log(1/\delta)}{\epsilon^2 a_{i, \ldots,i}^2}$ and $r  = 8\log(1/\delta)$.
\end{cor}

\begin{cor}
    Suppose $y_i^{(1)}, \ldots, y_i^{(K)}$ are the \text{i.i.d.} estimates of $a_{i, \ldots, i}$ obtained using $K$ different set of $\g^{(n)}$ for $n \in [N-1]$ in Theorem 5, where entries of $\g^{(n)}$ are \textit{i.i.d} $\mathcal{N}(0,1)$. Randomly divide the $K$ estimates into $r$ disjoint groups. The median-of-means of these $r$ groups yields an $(\epsilon,\delta)$ approximation  of $a_{i,\ldots,i}$ for $$K \geq \frac{32\left( \left(3^{N-1} -1 \right) \, a_{i, \ldots, i}^2 \, + \, 3^{N-2} \displaystyle \sum_{(j_1, \ldots, j_{N-1}) \in [d]^{N-1} \setminus \{(i, \ldots, i)\}} a_{j_1, \ldots, j_{N-1},i}^2\right) \log(1/\delta)}{\epsilon^2 a_{i, \ldots,i}^2}$$ and $r  = 8\log(1/\delta)$.
\end{cor}

\subsection{Trace estimator}
\label{section_trace_estimator}
The following theorem gives an unbiased estimator for the trace of a tensor and provides a bound on its variance. Its proof follows from the results of  Theorem ~\ref{thm:diag_rad_new}.

\begin{restatable}{theorem}{traceEstimatorNew}~\label{thm:trace_estimator_new}
   Let $\TA \in \R^{d \times \cdots \times d}$ be \anordertensor~with each \modeororder~size $d$.     
    Let $\g^{(n)} \in \R^d$ for $n \in [N-1]$ be  random vectors where entries are mean zero, have a unit second moment and finite fourth moment, and are pairwise independent, i.e. $\Expt{g^{(n)}_i} = 0, \Expt{\left(g^{(n)}_i\right)^2} = 1, \Expt{\left(g^{(n)}_i\right)^4} < \infty$, $\Expt{g^{(n)}_i g^{(m)}_j} = \Expt{g^{(n)}_i}\Expt{g^{(m)}_j} \forall~m \neq n$ or $i \neq j$. Let $\g := \g^{(1)} * \g^{(2)} * \cdots * \g^{(N-1)}$ and 
    \begin{align}
       \y := \g * \left(\TA \vtimes_1 \g^{(1)} \vtimes_2 \g^{(2)} \vtimes_3 \cdots \vtimes_{N-1} \g^{(N-1)} \right)
\end{align}
Then 
\begin{align}
X & := \g^T \left(\TA \vtimes_1 \g^{(1)} \vtimes_2 \g^{(2)} \vtimes_3 \cdots \vtimes_{N-1} \g^{(N-1)} \right)=\vec{1}^T\y =  \sum_{p=1}^d y_p \label{eq:eqn_est_tr_rad}
    \end{align}
    gives an unbiased estimate of the trace of tensor $\TA$, i.e. $\En{X}  = \trn{\TA}$    with variance
    \begin{align}
        \Varn{X} 
        & = \sum_{p=1}^d\left( \sum_{s=0}^{N-1} \Expt{\RVar^4}^s \left(  \sum_{\scriptsize
\begin{array}{c}
 (j_1, \ldots, j_{N-1})\\
\text{where $s$ of}~j_t, t \in [N-1]\\
\text{are equal to $i$}
\end{array}}    a_{j_1, \ldots, j_{N-1},p}^2   \right)  - a_{p, \ldots, p}^2 \right) \notag\\
 & + 2 \sum_{p>q}^d\left( \sum_{\scriptsize\begin{array}{c}
    (j_1,\hdots,j_{N-1}) \in \{p,q\} \\
    \text{and}~(k_1,\hdots,k_{N-1}) \in \{p,q\} \\
    \text{and}~j_t \neq k_t ~\forall~ t\in [N-1]
    \end{array}} a_{j_1,\hdots,j_{N-1},p}a_{k_1,\hdots,k_{N-1},q} - a_{p,\hdots,p}a_{q,\hdots,q}.     \right).
        \label{eq:eqn_theor_var_rad}
    \end{align}
where $\RVar$ is a random variable that has a distribution identical to the entries of $\g^{(n)}$.
\end{restatable}
\begin{proof}
    From Equation~\eqref{eq:eqn_est_tr_rad}, we have
    \begin{align}
        X & := \g^T \left(\TA \vtimes_1 \g^{(1)} \vtimes_2 \g^{(2)} \vtimes_3 \cdots \vtimes_{N-1} \g^{(N-1)} \right)=  \sum_{p=1}^d y_p. \notag
    \end{align}
    The expected value of $X$ is
    \begin{align}
        \E[X] & = \E\left[ \sum_{p=1}^d y_p\right] = \sum_{p=1}^d \E[y_p]  = \sum_{p=1}^d a_{p, \ldots, p} = \trn{A}. \numberthis \label{eq:eqn_tr_A}
    \end{align}
    We compute the variance of $X$ as follows
    \begin{align}
        \Varn{X} &= \Var\left(\sum_{p=1}^d y_p\right) \notag\\
        &= \sum_{p=1}^d \Varn{y_p} + 2\sum_{p > q}^d\Cov\left(y_p,y_q \right) \label{eq:100924_1} \\
        & = \sum_{p=1}^d\left( \sum_{s=0}^{N-1} \Expt{\RVar^4}^s \left(  \sum_{\scriptsize
\begin{array}{c}
 (j_1, \ldots, j_{N-1})\\
\text{where $s$ of}~j_t, t \in [N-1]\\
\text{are equal to $i$}
\end{array}}    a_{j_1, \ldots, j_{N-1},p}^2   \right)  - a_{p, \ldots, p}^2 \right) \notag \\
 & + 2 \sum_{p>q}^d\left( \sum_{\scriptsize\begin{array}{c}
    (j_1,\hdots,j_{N-1}) \in \{p,q\} \\
    \text{and}~(k_1,\hdots,k_{N-1}) \in \{p,q\} \\
    \text{and}~j_t \neq k_t~\forall~ t \in [N-1]
    \end{array}} a_{j_1,\hdots,j_{N-1},p}a_{k_1,\hdots,k_{N-1},q} - a_{p,\hdots,p}a_{q,\hdots,q}    \right).
        \label{eq:eqn_theor_var_rad_appendix}
    \end{align}
    Equation~\eqref{eq:eqn_theor_var_rad_appendix} holds due to Equations~\eqref{eq:310324_1}, \eqref{eq:310324_2} and \eqref{eq:100924_1}. Equations~\eqref{eq:eqn_tr_A} and \eqref{eq:eqn_theor_var_rad_appendix} completes a proof of the theorem.
\end{proof}

The following corollaries provide bounds on the variance of the trace estimator when the elements of the random vector $\g^{(n)}$, for $n \in [N-1]$, are from \textit{i.i.d.}   Rademacher and Gaussian distributions. They also give the bounds on the number of samples required to achieve $(\epsilon, \delta)$ estimator.  

\begin{restatable}{cor}{trRadCor} \label{cor:trace_rademacher_new}
If   the entries of $\g^{(n)}$, for $n \in [N-1]$,  in Theorem~\ref{thm:trace_estimator_new} are \textit{i.i.d.} Rademacher, then
\begin{align}
        \Varn{X} 
        & = \|\TA\|_F^2 - \tr\left({\TA}\right)^2  +2 \sum_{p>q}^d\left( \sum_{\scriptsize\begin{array}{c}
    (j_1,\hdots,j_{N-1}) \in \{p,q\} \\
    \text{and}~(k_1,\hdots,k_{N-1}) \in \{p,q\} \\
    \text{and}~j_t \neq k_t ~\forall~ t \in [N-1]
    \end{array}} a_{j_1,\hdots,j_{N-1},p}a_{k_1,\hdots,k_{N-1},q}    \right) \label{eq:eqn_symptotic_tr_rad}\\
    & \leq 2\left( \|\TA\|_F^2 - \sum_{j=1}^d a_{j,\hdots,j}^2\right). \label{eq:eqn_tr_rad_var_proof}
    \end{align}
    Further,  the mean of the  $K$ \textit{i.i.d.} estimates obtained using distinct sets of $\g^{(n)}$'s  gives an $(\epsilon, \delta)$-approximation of $tr(\TA)$ for
    
   $$K \geq O\left(\frac{ \left( \|\TA\|_F^2 - \sum_{j=1}^d a_{j,\ldots,j}^2 \right) \left( 2+\log(1/\delta)\right)^{2(N-1)}}{\epsilon^2 \cdot \tr(\TA)^2}\right).$$ 
\end{restatable}

\begin{proof}
We know that the fourth moment of the Rademacher random variable is $1$, which implies $\Expt{\RVar^4}=1$.  Hence, from Equation~\eqref{eq:eqn_theor_var_rad} of Theorem~\ref{thm:trace_estimator_new}, we have

\begin{align*}
    \Var(X) &=  \sum_{p=1}^d\left( \sum_{s=0}^{N-1}  \left(  \sum_{\scriptsize
\begin{array}{c}
 (j_1, \ldots, j_{N-1})\\
\text{where $s$ of}~j_t, t \in [N-1]\\
\text{are equal to $i$}
\end{array}}    a_{j_1, \ldots, j_{N-1},p}^2   \right)  - a_{p, \ldots, p}^2 \right) \notag \\
 & + 2 \sum_{p>q}^d\left( \sum_{\scriptsize\begin{array}{c}
    (j_1,\hdots,j_{N-1}) \in \{p,q\} \\
    \text{and}~(k_1,\hdots,k_{N-1}) \in \{p,q\} \\
    \text{and}~ j_t \neq k_t ~\forall~ t\in[N-1]
    \end{array}} \hspace{-0.3cm} a_{j_1,\hdots,j_{N-1},p}a_{k_1,\hdots,k_{N-1},q} - a_{p,\hdots,p}a_{q,\hdots,q}    \right)\\
    &= \sum_{p=1}^d \sum_{s=0}^{N-1}  \left(  \sum_{\scriptsize
\begin{array}{c}
 (j_1, \ldots, j_{N-1})\\
\text{where $s$ of}~j_t, t \in [N-1]\\
\text{are equal to $i$}
\end{array}}    a_{j_1, \ldots, j_{N-1},p}^2   \right)  - \sum_{p=1}^d a_{p, \ldots, p}^2  \notag \\
 & + 2 \sum_{p>q}^d \sum_{\scriptsize\begin{array}{c}
    (j_1,\hdots,j_{N-1}) \in \{p,q\} \\
    \text{and}~(k_1,\hdots,k_{N-1}) \in \{p,q\} \\
    \text{and}~j_t \neq k_t ~\forall~ t \in [N-1]
    \end{array}} 
\hspace{-0.4cm} a_{j_1,\hdots,j_{N-1},p} a_{k_1,\hdots,k_{N-1},q}
 - 2 \sum_{p>q}^d a_{p,\hdots,p} a_{q,\hdots,q}\\ 
    &= \|\TA\|_F^2 - \tr\left({\TA}\right)^2 +2 \sum_{p>q}^d \sum_{\scriptsize\begin{array}{c}
    (j_1,\hdots,j_{N-1}) \in \{p,q\} \\
    \text{and}~(k_1,\hdots,k_{N-1}) \in \{p,q\} \\
    \text{and}~j_t \neq  k_t ~\forall~ t \in [N-1]
    \end{array}} \hspace{-0.4cm} a_{j_1,\hdots,j_{N-1},p}a_{k_1,\hdots,k_{N-1},q}  \numberthis \label{eq:121025_1}\\
    &\leq 2 \left( \|\TA\|_F^2 - \sum_{j=1}^d a_{j, \ldots,j}^2\right). \numberthis \label{eq:121025_2}
\end{align*}
The Equations~\eqref{eq:121025_1} and \eqref{eq:121025_2} hold due to the following facts
\begin{align*}
    \sum_{p=1}^d\left(\sum_{j_1,\hdots,j_{N-1}}a^2_{j_1,\hdots,j_{N-1},p}\right) = \|\TA\|_F^2,
\end{align*}
\begin{align*}
    \sum_{p=1}^d a_{p,\hdots,p}^2 +2 \sum_{p>q}a_{p,\hdots,p}a_{q,\hdots,q} = \left(\sum_{p=1}^d a_{p,\hdots,p}\right)^2 = \tr\left(\mathcal A\right)^2
\end{align*}
and 
\begin{align*}
a_{j_1,\hdots,j_{N-1},p}^2 + a_{k_1,\hdots,k_{N-1},q}^2 \geq 2a_{j_1,\hdots,j_{N-1},p} a_{k_1,\hdots,k_{N-1},q}. 
\end{align*}
    Let $\bar{X}$ denotes the average of $K$ $\text{i.i.d.}$ estimates of the $\tr(\TA)$ obtained using distinct sets of $\g^{(n)}$'s for $n \in [N-1]$. Since the estimates are \textit{i.i.d.}, we have
\begin{align}
    \Var(\bar{X}) \leq \frac{2\left( \| \TA\|_F^2 - \sum_{j=1}^d a_{j, \ldots, j}^2\right)}{K}. 
\end{align}
The variance of $\bar{X}$ is bounded, and $\bar{X}$ is a polynomial of degree $2(N-1)$ of independent Rademacher random variables (the entries corresponding to distinct sets of $\g^{(n)}$'s for $n \in [N-1]$). Then, for some absolute constant $R$ by utilizing the  Hypercontractivity Concentration Inequality stated in Thoerem~\ref{thm:hpercontractractivity}, we have
\begin{align}
    &\Pr\left(|\bar{X} -  \tr(\TA)| \geq \epsilon \cdot \tr(\TA) \right) \leq  e^2 \cdot e^{-\left( \frac{\epsilon^2 \tr(\TA)^2}{R \cdot \Var(\bar{X})}\right)^{\frac{1}{2(N-1)}}} \notag\\
    &= e^2 \cdot  e^{-\left( \frac{\epsilon^2 \cdot  \tr(\TA)^2 \cdot K}{R \cdot 2 \cdot \left( \|\TA\|_F^2 - \sum_{j=1}^d a_{j, \ldots,j}^2 \right)}\right)^{\frac{1}{2(N-1)}}} \notag\\
    & \leq \delta ~\left(\text{if we choose } K \geq  \frac{2 R \left( \|\TA\|_F^2 - \sum_{j=1}^d a_{j,\ldots,j}^2 \right) \qquad \left( 2+\log(1/\delta)\right)^{2(N-1)}}{\epsilon^2 \cdot \tr(\TA)^2} \right. \notag\\
    &\hspace{10cm} \text{ in above equation} \Bigg).  \notag
\end{align}
\end{proof}

\begin{restatable}{cor}{trNorCor}  \label{cor:trace_normal_new}
If the entries of $\g^{(n)}$, for $n \in [N]$, in Theorem~\ref{thm:trace_estimator_new} are \textit{i.i.d.} $\mathcal N(0,1)$, then
\begin{align}
        \Varn{X} & =\sum_{p=1}^d\left( \sum_{s=0}^{N-1} 3^s \left(  \sum_{\scriptsize
\begin{array}{c}
 (j_1, \ldots, j_{N-1})\\
\text{where $s$ of}~j_t, t \in [N-1]\\
\text{are equal to $i$}
\end{array}}    a_{j_1, \ldots, j_{N-1},p}^2   \right)  - a_{p, \ldots, p}^2 \right) \notag\\
 & + 2 \sum_{p>q}^d\left( \sum_{\scriptsize\begin{array}{c}
    (j_1,\hdots,j_{N-1}) \in \{p,q\} \\
    \text{and}~(k_1,\hdots,k_{N-1}) \in \{p,q\} \\
    \text{and}~j_t \neq k_t~\forall~ t \in [N-1]
    \end{array}} a_{j_1,\hdots,j_{N-1},p}a_{k_1,\hdots,k_{N-1},q} - a_{p,\hdots,p}a_{q,\hdots,q} \right) \label{eq:eqn_asymp_tr_normal} \\
    & \leq (3^{N-1}-1)\|\TA||_F^2.
        \label{eq:eqn_tr_normal_var_proof}
    \end{align}
Further, the mean of the  $K$ \textit{i.i.d.} estimates obtained using distinct sets of $\g^{(n)}$'s gives an $(\epsilon, \delta)$ approximation of $tr(\TA)$ for 
$$K \geq O\left( \frac{ \left(3^{N-1} -1 \right) \|\TA\|_F^2  \left( 2 + \log(1/\delta)\right)^{2(N-1)}}{ \left({\epsilon^2 \cdot \tr(\TA)^2}\right)}\right).$$
\end{restatable}

\begin{proof} The fourth moment of the standard normal distribution is $3$.  So, from Equation~\eqref{eq:eqn_theor_var_rad} of Theorem~\ref{thm:trace_estimator_new}, we have
\begin{align*}
   & \Var(X) =  \sum_{p=1}^d\left( \sum_{s=0}^{N-1}  3^s\left(  \sum_{\scriptsize
\begin{array}{c}
 (j_1, \ldots, j_{N-1})\\
\text{where $s$ of}~j_t, t \in [N-1]\\
\text{are equal to $i$}
\end{array}}    a_{j_1, \ldots, j_{N-1},p}^2   \right)  - a_{p, \ldots, p}^2 \right) \notag \\
 &\qquad \quad + 2 \sum_{p>q}^d\left( \sum_{\scriptsize\begin{array}{c}
    (j_1,\hdots,j_{N-1}) \in \{p,q\} \\
    \text{and}~(k_1,\hdots,k_{N-1}) \in \{p,q\} \\
    \text{and}~j_t \neq k_t ~\forall~ t \in [N-1]
    \end{array}} a_{j_1,\hdots,j_{N-1},p}a_{k_1,\hdots,k_{N-1},q} - a_{p,\hdots,p}a_{q,\hdots,q}    \right)\\
 & = 3^{N-1} \sum_{p=1}^d a_{p,\ldots,p}^2 + \sum_{p=1}^d\sum_{s=0}^{N-2}  3^s \left(  \sum_{\scriptsize
\begin{array}{c}
 (j_1, \ldots, j_{N-1})\\
\text{where $s$ of}~j_t, t \in [N-1]\\
\text{are equal to $i$}
\end{array}} \hspace{-0.5cm}   a_{j_1, \ldots, j_{N-1},i}^2   \right)  - \sum_{p=1}^d a_{i, \ldots, i}^2 \notag\\
& \qquad \quad + 2 \sum_{p>q}\sum_{\scriptsize\begin{array}{c}
    (j_1,\hdots,j_{N-1}) \in \{p,q\} \\
    \text{and}~(k_1,\hdots,k_{N-1}) \in \{p,q\} \\
    \text{and}~ j_t \neq k_t ~\forall ~ t \in [N-1]
    \end{array}} a_{j_1,\hdots,j_{N-1},p}a_{k_1,\hdots,k_{N-1},q} - 2 \sum_{p>q} a_{p,\hdots,p}a_{q,\hdots,q}\\ 
&\leq 3^{N-1} \sum_{p=1}^d a_{p,\ldots,p}^2 + 3^{N-2}\sum_{p=1}^d\sum_{s=0}^{N-2}   \left(  \sum_{\scriptsize
\begin{array}{c}
 (j_1, \ldots, j_{N-1})\\
\text{where $s$ of}~j_t, t \in [N-1]\\
\text{are equal to $i$}
\end{array}} \hspace{-0.3cm}   a_{j_1, \ldots, j_{N-1},i}^2   \right)  - \sum_{p=1}^d a_{i, \ldots, i}^2 \notag\\
&  \qquad + 2 \sum_{p>q}\sum_{\scriptsize\begin{array}{c}
    (j_1,\hdots,j_{N-1}) \in \{p,q\} \\
    \text{and}~(k_1,\hdots,k_{N-1}) \in \{p,q\} \\
    \text{and}~ j_t \neq k_t ~\forall~ t \in [N-1]
    \end{array}} a_{j_1,\hdots,j_{N-1},p}a_{k_1,\hdots,k_{N-1},q} - 2 \sum_{p>q} a_{p,\hdots,p}a_{q,\hdots,q}\\ 
     \begin{split}
          &= \left(3^{N-1} -1 \right) \sum_{p=1}^{d} a_{p,\ldots, p}^2 + 3^{N-2} \sum_{(j_1, \ldots, j_{N}) \in [d]^N \setminus \{(p, \ldots,p)| p \in [d]\} } a_{j_1, \ldots,  j_N}^2   \\
     &\phantom{=}\qquad + 2 \sum_{p>q}\sum_{\scriptsize\begin{array}{c}
    (j_1,\hdots,j_{N-1}) \in \{p,q\} \\
    \text{and}~(k_1,\hdots,k_{N-1}) \in \{p,q\} \\
    \text{s.t.}~ j_t  \neq k_t ~\forall~ t \in [N-1],\\
    (j_1, \ldots, j_{N-1}) \neq (p, \ldots,p)\\
    \text{and } (k_1, \ldots, k_{N-1}) \neq (q, \ldots,q)
    \end{array}} a_{j_1,\hdots,j_{N-1},p}a_{k_1,\hdots,k_{N-1},q}\\
    \end{split}\\
    \begin{split}
          &\leq \left(3^{N-1} -1 \right) \sum_{p=1}^{d} a_{p,\ldots, p}^2 + 3^{N-2} \sum_{(j_1, \ldots, j_{N}) \in [d]^N \setminus \{(p, \ldots,p)| p \in [d]\} } a_{j_1, \ldots,  j_N}^2   \\
     &\phantom{=}\qquad +  \sum_{p>q}\sum_{\scriptsize\begin{array}{c}
    (j_1,\hdots,j_{N-1}) \in \{p,q\} \\
    \text{and}~(k_1,\hdots,k_{N-1}) \in \{p,q\} \\
    \text{s.t.}~ j_t  \neq k_t ~\forall~ t \in [N-1],\\
    (j_1, \ldots, j_{N-1}) \neq (p, \ldots,p)\\
    \text{and } (k_1, \ldots, k_{N-1}) \neq (q, \ldots,q)
    \end{array}} a_{j_1,\hdots,j_{N-1},p}^2 + a_{k_1,\hdots,k_{N-1},q}^2\\
    \end{split}\\
    \begin{split}
          &\leq \left(3^{N-1} -1 \right) \sum_{p=1}^{d} a_{p,\ldots, p}^2 + 3^{N-2} \sum_{(j_1, \ldots, j_{N}) \in [d]^N \setminus \{(p, \ldots,p)| p \in [d]\} } a_{j_1, \ldots,  j_N}^2 \notag\\
          &\phantom{=}\qquad + \left(\|\TA\|_F^2 - \sum_{j=1}^d a_{j, \ldots, j}^2 \right)\\
    \end{split}\\
    & = \left( 3^{N-1} -2 \right) \sum_{j=1}^d a_{j, \ldots, j}^2  + 3^{N-2} \sum_{(j_1, \ldots, j_{N}) \in [d]^N \setminus \{(p, \ldots,p)| p \in [d]\} } a_{j_1, \ldots,  j_N}^2 + \| \TA \|_{F}^2 \nonumber\\
    & \leq  \left( 3^{N-1} -2 \right) \left(  \sum_{j=1}^d a_{j, \ldots, j}^2 + \sum_{(j_1, \ldots, j_{N}) \in [d]^N \setminus \{(p, \ldots,p)| p \in [d]\} } \right) + \| \TA \|_{F}^2 \nonumber\\
     & = \left( 3^{N-1} -2 \right) \|\TA\|_{F}^2 + \| \TA \|_{F}^2 \nonumber\\
    & = \left( 3^{N-1} - 1 \right) \|\TA\|_{F}^2. \numberthis \label{eq:eqn_tr_normal_var_proof_appendix}
\end{align*}
We can easily prove the concentration bound by utilizing the  Hypercontractivity Concentration Inequality stated in Theorem~\ref{thm:hpercontractractivity} and employing
the same steps as used in the proof of Corollary~\ref{cor:trace_rademacher_new}.
\end{proof}

If the $(\epsilon, \delta)$ estimator is not required to be linear, the exponential dependence on the tensor order $N$ in the term $\log(1/\delta)$ appearing in the sample complexity bounds in the above corollaries can be eliminated using the \textit{median-of-means} estimator. The following corollaries provide bounds with improved dependence on $\delta$ by leveraging Lemma~\ref{lem:median_of_means}.

\begin{cor}
    Let $X_1, \ldots, X_{K}$ be the \text{i.i.d.} estimates of $\trn{\TA}$ obtained using $K$ different set of $\g^{(n)}$ for $n \in [N-1]$, where entries of $\g^{(n)}$ are \textit{i.i.d} Rademacher. Divide the $K$ estimates randomly into $r$ disjoint groups. The median-of-means of these $r$ groups gives an $(\epsilon,\delta)$ approximation  for $\trn{\TA}$ for $K \geq \frac{64\left( \|\TA\|_F^2 - \sum_{j=1}^d a_{j,\hdots,j}^2\right) \log(1/\delta)}{\epsilon^2 \trn{\TA}^2}$ and $r  = 8\log(1/\delta)$.
\end{cor}

\begin{cor}
    Let $y_i^{(1)}, \ldots, y_i^{(K)}$ be the \text{i.i.d.} estimates of $a_{i, \ldots, i}$ obtained using $K$ different set of $\g^{(n)}$ for $n \in [N-1]$, where entries of $\g^{(n)}$ are \textit{i.i.d} $\mathcal{N}(0,1)$. Randomly divide the $K$ estimates into $r$ disjoint groups. The median-of-means of these $r$ groups yields an $(\epsilon,\delta)$ approximation  of $\trn{\TA}$ for $K \geq \frac{32 \left(3^{N-1} -1 \right) \|\TA\|_F^2 \log(1/\delta)}{\epsilon^2 \trn{\TA}^2}$ and $r = 8\log(1/\delta)$.
\end{cor}

\noindent \textbf{Comment on the tightness of the variance upper bound given in Corollaries~\ref{cor:trace_rademacher_new} and \ref{cor:trace_normal_new}:} 
 Equations~\eqref{eq:eqn_symptotic_tr_rad} and \eqref{eq:eqn_asymp_tr_normal} give the exact variance of our trace estimation proposal when the elements of $\g^{(n)}$ are sampled from the Rademacher and Gaussian distribution, respectively. We upper bound it in Equation~\eqref{eq:eqn_tr_rad_var_proof}, and \eqref{eq:eqn_tr_normal_var_proof} to obtain their clean and interpretable expression, which we subsequently used for giving concentration bounds. We perform numerical simulations to understand the tightness of these upper bounds. We show it for  Rademacher distribution. We experimentally compute the ratio of the upper bound and the exact variance expression by computing these terms for the tensors having all entries $1$, with different combinations of $N$ and $d$.  Figure~\ref{fig:upper_bound} presents the graph of the ratio of exact variance and upper bound w.r.t. mode size ($d$) for different values of $N$. From Figure~\ref{fig:upper_bound}, we observe that the variance's upper bound stated in Equation~\eqref{eq:eqn_tr_rad_var_proof} becomes loose as we increase the value of $N$ and $d$.
  \begin{figure}
  \centering
     \includegraphics[width=0.55\textwidth]{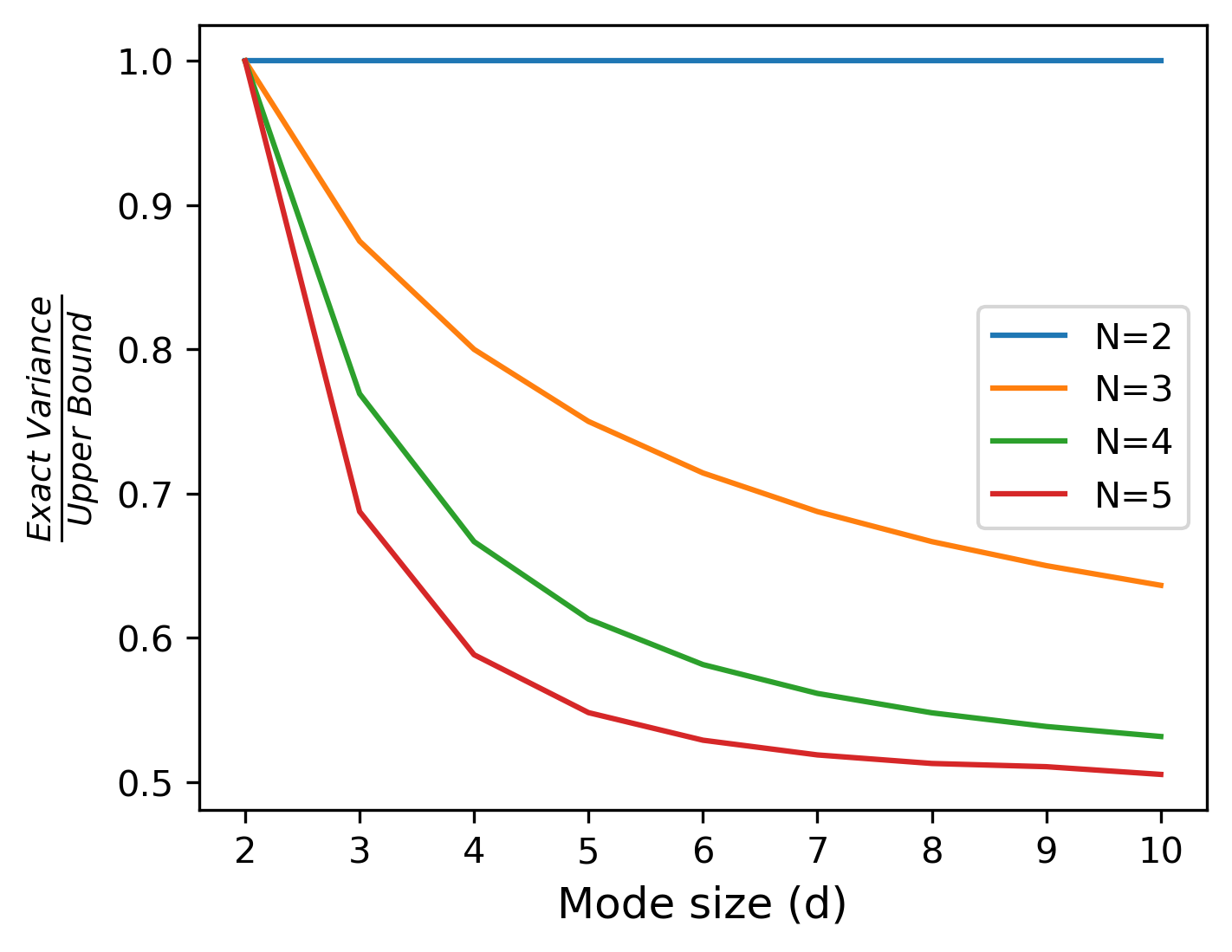}
     \captionof{figure}{Analysis of the tightness of the upper bound (Equation~\eqref{eq:eqn_tr_rad_var_proof}) w.r.t. the exact variance expression (Equation~\eqref{eq:eqn_symptotic_tr_rad}) for different combinations of $N$ and $d$. $N$ denotes the order of the tensor. The exact variance to upper bound ratio being closer to $1$ means the upper bound is tight.}
     \label{fig:upper_bound}
  \end{figure}
 
\section{Experimental Results} \label{sec:experiment}
We complement the theoretical analysis of our proposals via supporting experiments on synthetic datasets. Our experimental results also validate that that the variance of the Rademacher distribution-based diagonal estimator is smaller than that of the Normal distribution based estimator.
 We generate our datasets as follows: let $\alpha$ denote the ratio of the sum of squares of diagonal entries and squared Frobenius norm, and $N$ denote the order of the tensor. We randomly generate tensors for different values of $\alpha$ and $N$, while keeping the dimension along each mode as $100$.\\

\noindent \textbf{Experimental Setup:} In our experiments, we choose $N \in \{2,3,4\}$ and $\alpha \in \{0.2,0.4,0.6,0.8\}$. For each combination of $N$ and $\alpha$, we compute the $K$ \textit{i.i.d.} estimate of diagonal entries and trace using our proposals (Definition~\ref{def:diag_rad} and \ref{def:tr_Rad_tr} respectively) and consider their average as a representative estimate. That is, we take the mean of $K$ tensor-vector queries as the representative estimate.  To evaluate the quality of the diagonal estimate, we calculate the \emph{absolute relative errors} using the following formula: $\left| \frac{\bar{y}_i - a_{i, \ldots, i}}{a_{i, \ldots, i}} \right|$, where $\bar{y}_i$ denotes the average of $K$ \textit{i.i.d.} estimates of diagonal element $a_{i,\ldots, i}$ obtained using our diagonal estimator Definition~\ref{def:diag_rad}. Further, to evaluate the quality of the trace estimate, we use the following expression: $\left| \frac{\bar{X} - \tr(\TA)}{\tr(\TA) } \right|$, where $\bar{X}$ denotes the average of $K$ \textit{i.i.d.} estimates of $\tr(\TA)$ obtained using our trace estimator proposal Definition~\ref{def:tr_Rad_tr}.  In our experiments, we use  $K \in \{2,4,6,8,10,12,14,16,18,20\}$.

Our experimental study considers the diagonal entries and trace estimator using both Rademacher and Normal distribution. We present their comparison based on the Mean of the Absolute Relative Errors (MARE) observed over $100$ independent experimental runs. A smaller value of the MARE is an indication of a better estimate. Further, we analyze the variability of the estimators by generating boxplots of the \emph{signed relative errors} from the same 100 runs. A smaller interquartile range in the boxplot indicates lower variance and provides further insight about the consistency of the estimators. The \emph{signed relative error} for the diagonal estimates is defined as $\frac{\bar{y}_i - a_{i,\ldots,i}}{|a_{i,\ldots,i}|}$, where \( \bar{y}_i \) denotes the average of \( K \) i.i.d.\ estimates of the \( i \)-th diagonal entry, and \( a_{i,\ldots,i} \) is the corresponding true value. Similarly, the \textit{signed relative error} for the trace estimate is defined as $\frac{\bar{X} - \tr(\mathcal{A})}{|\tr(\mathcal{A})|}$, where \( \bar{X} \) is the average of \( K \) i.i.d.\ trace estimates, and \( \tr(\mathcal{A}) \) is the true trace of the tensor.\\ 

\textbf{Diagonal Estimator:} We summarise experimental observations on our diagonal estimator in Figures~\ref{fig:diag_boxplot_3_cross_4} and \ref{fig:diag_mare_3_cross_4}. Figure~\ref{fig:diag_mare_3_cross_4}  depicts the comparison of Rademacher and Normal distribution-based diagonal estimators for a randomly chosen diagonal element based on mean absolute relative error, over $100$ experimental runs for $N \in \{2,3,4\}$ and $\alpha \in \{0.2,0.4,0.6,0.8\}$. Figure~\ref{fig:diag_boxplot_3_cross_4} presents the variance analysis of the relative errors via boxplots observed over $100$ runs for Rademacher and Normal distribution-based diagonal estimators for a randomly chosen diagonal element.
\begin{figure}[t]
    \centering
    \includegraphics[width=0.95\textwidth]{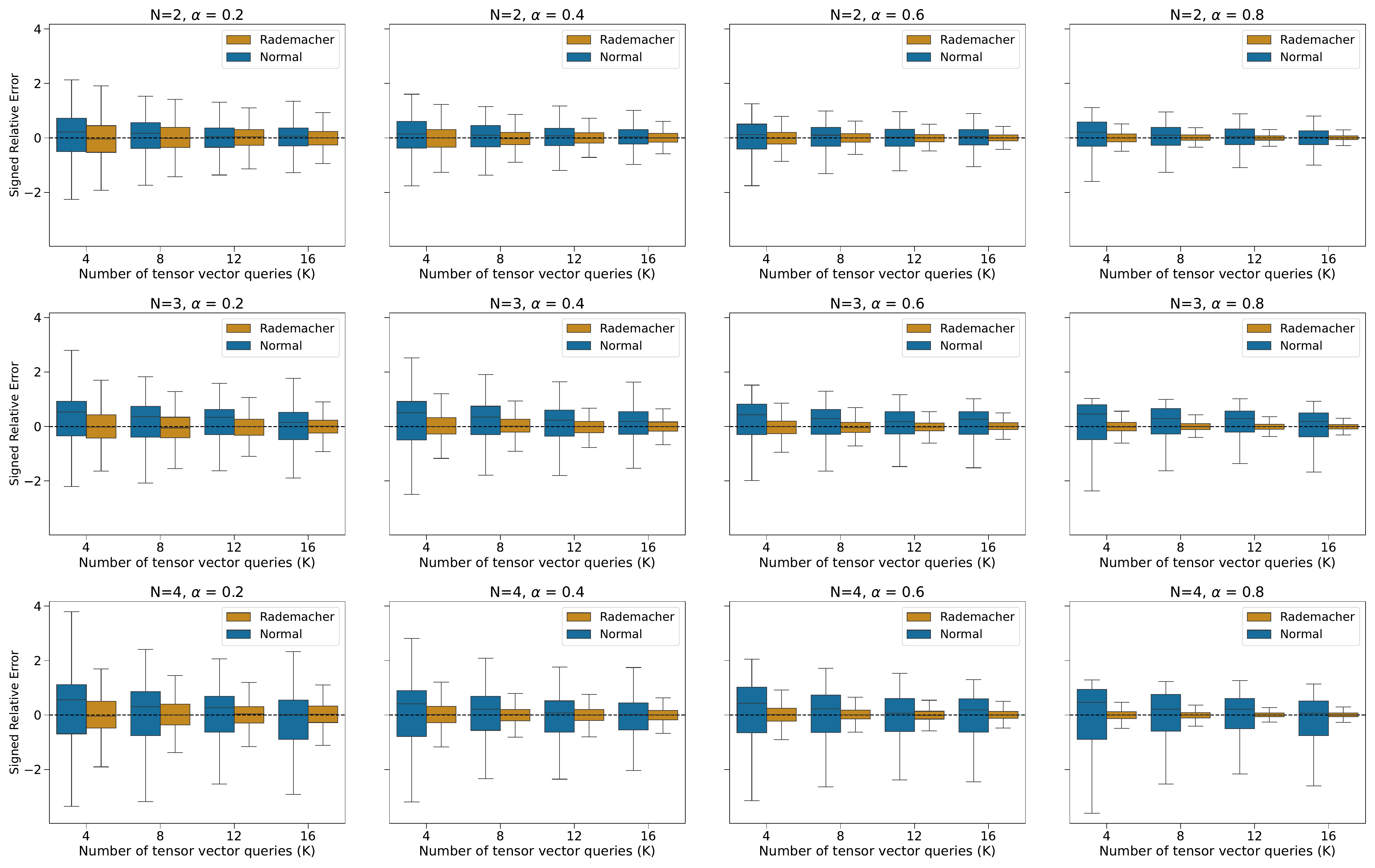}
    \caption{Variance analysis of   Rademacher and Normal distribution based diagonal estimators via boxplots using the relative errors observed in  $100$ runs. $N$ and $\alpha$ denote the order of the tensor and the ratio of the sum of the square of diagonal entries and the square Frobenius norm of the tensor, respectively. The smaller interquartile range is an indication of a smaller variance.}
    \label{fig:diag_boxplot_3_cross_4}
\end{figure}
\begin{figure}[t]
    \centering
    \includegraphics[width=0.95\textwidth]{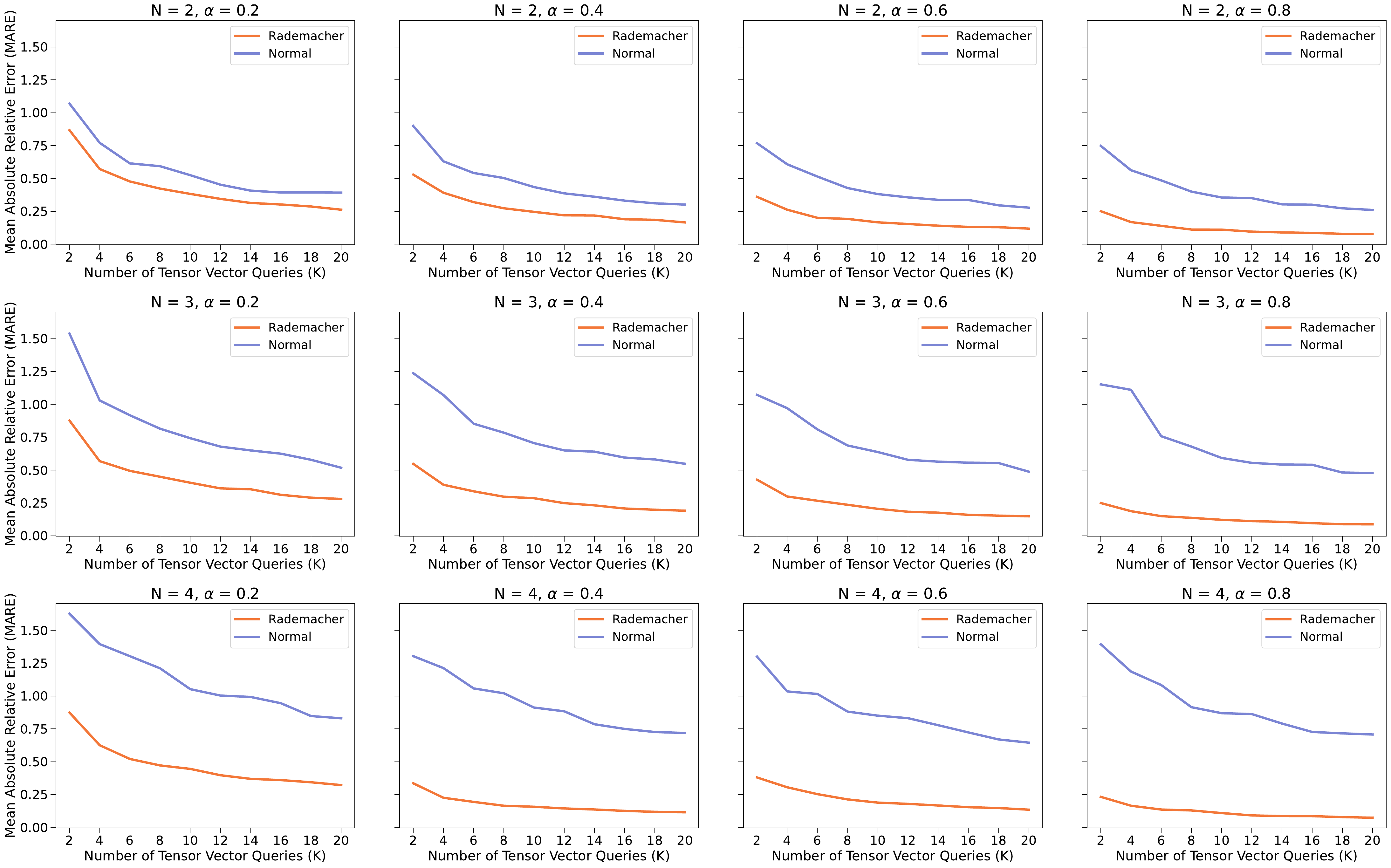}
    \caption{Comparison among the Rademacher and the normal distribution based diagonal estimators on the basis of mean absolute relative error over $100$ experimental runs for $N \in \{2,3,4\}$. $N$ and $\alpha$ denote the order of the tensor and the ratio of the sum of the squares of diagonal entries and the square Frobenius norm of the tensor, respectively. The smaller value of the mean absolute relative error indicates better estimates.}
    \label{fig:diag_mare_3_cross_4}
\end{figure}

\textbf{Insight:}  From Figure~\ref{fig:diag_boxplot_3_cross_4}, it is evident that the interquartile range of the boxplots of the Rademacher distribution-based diagonal estimator is smaller than that of the Normal distribution-based estimator. This implies that the variance of the Rademacher distribution-based diagonal estimator is smaller than that of the Normal distribution based estimator.  The interquartile range of the boxplots associated with the Rademacher distribution-based diagonal estimator decreases as the value of $\alpha$ (the ratio of the sum of squares of diagonal entries to the Frobenius norm of the tensor) increases and remains independent of the value of $N$ (order of the tensor). This observation aligns with the theoretical bounds on variance stated in Equation~\eqref{eq:eqn_0410_1}. On the other hand, the interquartile range of the boxplots for the Normal distribution-based diagonal estimator increases with $N$ but remains independent of the value of $\alpha$. This also aligns with our theoretical expression in Equation~\eqref{eq:eqn_var_yi_Rad}, where the estimates' variance increases with $N$.
Similarly, from Figure~\ref{fig:diag_mare_3_cross_4}, it is evident that the Rademacher distribution-based diagonal estimator outperforms the corresponding Normal distribution-based diagonal estimator. Furthermore, we note that the Rademacher distribution-based estimator's Mean Absolute Relative Error (MARE) decreases as $\alpha$ increases and remains independent of $N$. In contrast, the MARE of the Normal distribution-based diagonal estimator increases with an increase in $N$ but remains independent of $\alpha$. These observations are in line with the observations related to the boxplots in Figure~\ref{fig:diag_boxplot_3_cross_4} and are consistent with the theoretical variance and concentration bounds of the respective estimators.

\textbf{Trace Estimator:} \label{subsection:exp_trace}
We summarise our experimental findings for trace estimation in Figures~\ref{fig:tr_boxplot} and \ref{fig:tr_mare}.  Figure~\ref{fig:tr_boxplot} presents the variance analysis of the relative errors observed over $100$ runs  via boxplots for $N \in \{2, 3,4\}$ and $\alpha \in \{0.2,0.4,0.6,0.8\}$.  Figure~\ref{fig:tr_mare} presents the comparison based on a mean absolute relative error, over $100$ experimental runs for Rademacher and Normal distribution-based trace estimators.
\begin{figure}[t]
    \centering
    \includegraphics[width=0.95\textwidth]{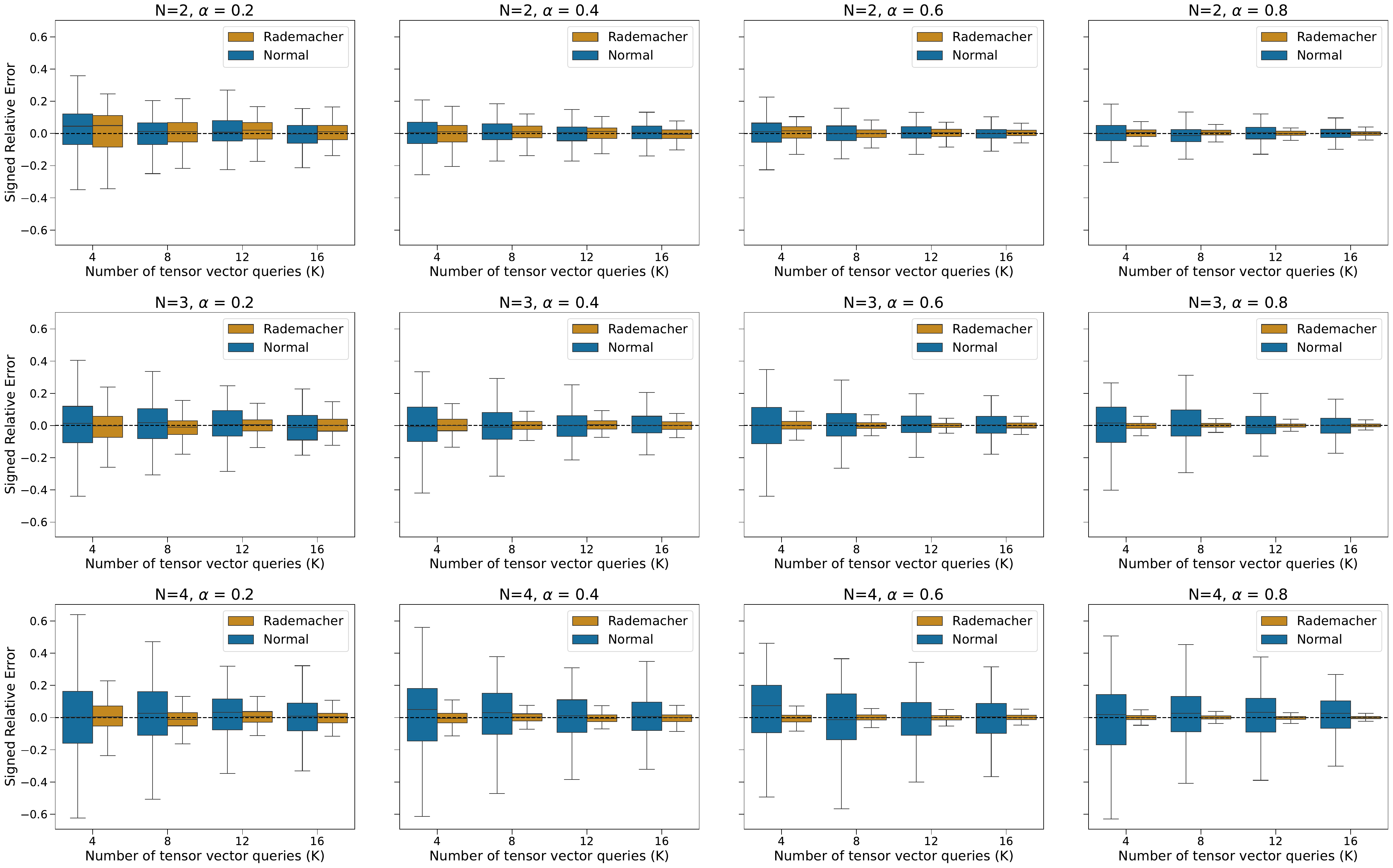}
    \caption{Variance analysis of   Rademacher and Normal distribution based trace estimators via boxplots using the relative errors observed in  $100$ runs. $N$ and $\alpha$ denote the order of the tensor and the ratio of the sum of the square of diagonal entries and the square Frobenius norm of the tensor, respectively. The smaller interquartile range is an indication of a smaller variance.}
    \label{fig:tr_boxplot}
\end{figure}

\textbf{Insight:}  From Figure~\ref{fig:tr_boxplot}, it is clear that the interquartile range of the boxplots of the Rademacher distribution-based trace estimator is smaller than that of the Normal distribution-based estimator, which implies that the variance of the Rademacher distribution-based trace estimator is smaller than that of Normal distribution based estimator. We observe the interquartile range of the boxplots corresponding to the Rademacher distribution-based trace estimator decreases as we increase the value of $\alpha$ (ratio of the sum of squares of the diagonal entries and the Frobenius norm of the tensor) and remains independent of the value of $N$ (order of the tensor). This aligns with our theoretical bounds on the variance stated in Equation~\eqref{eq:eqn_tr_rad_var_proof}. However, the interquartile range of the boxplots corresponding to the Normal distribution-based trace estimator increases with the value of $N$ and remains independent of the value of  $\alpha$. This also aligns with our theoretical expression in Equation~\eqref{eq:eqn_tr_normal_var_proof} where the estimates' variance increases with $N$.
\begin{figure}[t]
    \centering
    \includegraphics[width=0.95\textwidth]{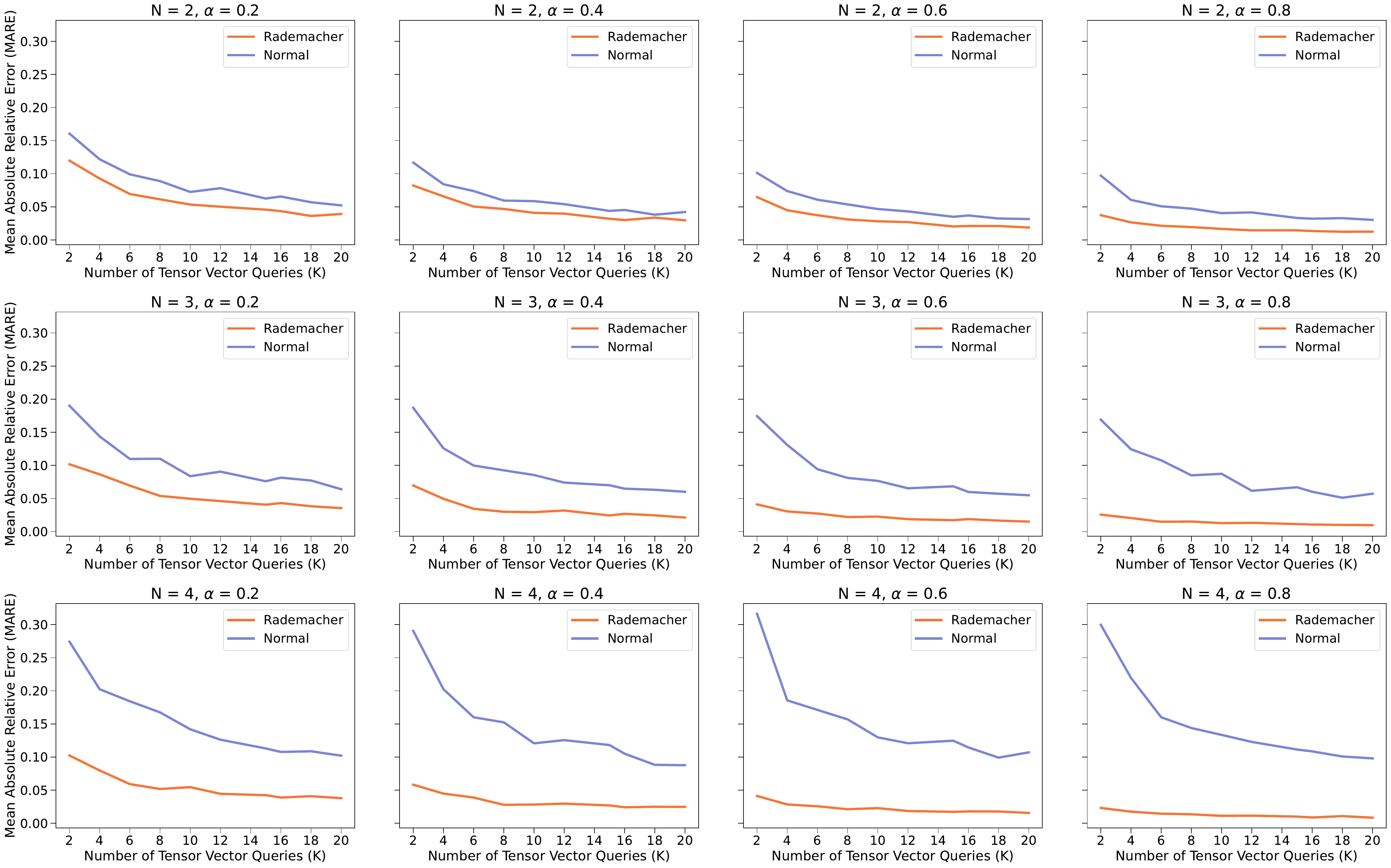}
    \caption{Comparison among the Rademacher and the normal distribution based trace estimators on the basis of mean absolute relative error over $100$ experimental runs for $N \in \{2,3,4\}$. $N$ and $\alpha$ denote the order of the tensor and the ratio of the sum of the square of diagonal entries and the square Frobenius norm of the tensor, respectively. The smaller value of mean absolute relative error indicates better estimates.}
    \label{fig:tr_mare}
\end{figure}

Similar to Figure~\ref{fig:tr_boxplot}, from Figure~\ref{fig:tr_mare}, it is also evident that the Rademacher distribution-based trace estimator performs better than the corresponding normal distribution-based trace estimator. Further, we observe that the MARE of the Rademacher distribution-based estimator decreases with an increase in the value of $\alpha$ and remains independent of the value of $N$. In contrast, the MARE of the Normal distribution-based trace estimator increases with an increase in the value of $N$ and remains independent of the value of $\alpha$. These insights support the observations corresponding to the boxplots (Figure~\ref{fig:tr_boxplot}) and align with the respective estimators' theoretical variance and concentration bounds.

\section{Conclusion \& open questions} \label{sec:conclusion}
 We proposed unbiased estimators for the trace and diagonal entries of higher-order tensors, under the tensor-vector multiplication queries model. Our proposals generalize the classical Hutchinson's trace~\citep{hutchinson1989stochastic}, and the diagonal elements estimators~\citep{bekas2007estimator} of matrices to higher order tensors as our estimators reduce to these estimators for $N=2$. We presented a theoretical analysis of our proposals and provided their $(\epsilon, \delta)$ estimators. 
Our proposals are simple, effective and easy to implement. We hope our proposals will benefit applications involving computing the trace or diagonal entries of higher-order tensors when tensor entries are accessed via tensor-vector queries. We state and give several open questions and research directions below.

\begin{inparaenum}[\itshape a)\upshape]
\item  One of the major research directions is to derive a tighter upper bound on the query complexity for the trace and diagonal estimators proposed in this work. Further, deriving a lower bound on the number of samples is also an interesting open question of the work. 
\item The second open question is how the structural properties of tensors, such as symmetry, sparsity or low-rankness, etc, can be exploited to design improved algorithms for trace and diagonal estimation, analogous to Hutch++ \citep{meyer2021hutch++}, XTRACE~\citep{epperly2024xtrace} and Diag++ \citep{baston2022stochastic} for matrices.  \item Another interesting research direction is improving the proposed estimators by leveraging variance reduction techniques such as control variate (CV) method and others suggested by \cite{adams2018estimating, frommer2022multilevel,frommer2023mg} for the matrix case.
\item We believe that our result will be beneficial in areas such as hypergraph spectral theory, quantum computing, and other domains where Hutchinson type estimators have been applied to matrices, but the underlying data is naturally tensor-structured. Thus, a valuable direction for future research is to explore and identify potential application areas where these techniques could provide practical benefits.
\end{inparaenum}

 \bibliographystyle{elsarticle-num-names} 
\bibliography{references}
\end{document}